\documentstyle[twoside,amsmath, amssymb, mathrsfs, amsfonts, eucal, epsfig, 11pt]{article}

 \newtheorem{theorem}{Theorem}[section]
 \newtheorem{lemma}[theorem]{Lemma}
 \newtheorem{proposition}[theorem]{Proposition}
 \newtheorem{corollary}[theorem]{Corollary}
 \newtheorem{definition}[theorem]{Definition}
 \newenvironment{proof}{\begin{trivlist} \item[]{\em Proof.}}{\end{trivlist}}
 \newcount\refno
\def\CC{{\mathbb C}}
\def\DD{{\mathbb D}}

 \def\RR{{\mathbb R}}
 \def\NN{{\mathbb N}}

 \title{\bf Boundedness of operators on the Bergman spaces associated with a class of generalized analytic functions
 \thanks{{Supported by the National Natural
 Science Foundation of China (No. 12071295).}
 \newline
 \indent \,\,$^\dag$Corresponding author.
 \newline
 \indent\,\, E-mail: lizk@shnu.edu.cn (Zh.-K. Li); hhwei@cslg.edu.cn (H.-H. Wei).}}

\author{Zhongkai Li$^{1}$ and Haihua Wei$^{\dag,2}$\\
{\small $^{1}$Department of Mathematics, Shanghai Normal University}\\
{\small Shanghai 200234, China} \\
{\small $^{2}$School of Mathematics and Statistics, Changshu Institute of Technology}\\
{\small Changshu 215500, Jiangsu, China}
}

\date{}

\begin{document}
 \maketitle \setcounter{page}{1} \pagestyle{myheadings}
 \markboth{Li and Wei}{Operators on Bergman spaces}

 \begin{abstract}
 \noindent
The purpose of the paper is to study the operators on the weighted Bergman spaces on the unit disk ${\mathbb{D}}$, denoted by $A^{p}_{\lambda,w}({\mathbb{D}})$, that are associated with a class of generalized analytic functions, named the $\lambda$-analytic functions, and with a class of radial weight functions $w$. For $\lambda\ge0$, a $C^2$ function $f$ on ${{\mathbb D}}$ is said to be
$\lambda$-analytic if $D_{\bar{z}}f=0$, where $D_{\bar{z}}$ is the (complex) Dunkl operator given by  $D_{\bar{z}}f=\partial_{\bar{z}}f-\lambda(f(z)-f(\bar{z}))/(z-\bar{z})$.
It is shown that, for $2\lambda/(2\lambda+1)\le p\le1$, the boundedness of an operator from $A^{p}_{\lambda,w}({\mathbb{D}})$ into a Banach space depends only upon the norm estimate of a single vector-valued $\lambda$-analytic function. As applications, we obtain a necessary and sufficient conditions of sequence multipliers on the spaces $A^{p}_{\lambda,w}({\mathbb{D}})$ for general weights $w$, and characterize the dual space of $A^{p}_{\lambda,w}({\mathbb{D}})$ for the power weight $w=(1-|z|^2)^{\alpha-1}$ with $\alpha>0$, and also give a sufficient condition of Carleson type for boundedness of multiplication operators on $A^{p}_{\lambda,w}({\mathbb{D}})$.

 \vskip .2in
 \noindent
 {\bf 2020 MS Classification:} 30H20, 47B91 (Primary), 30G30 (Secondary)
 \end{abstract}

\setcounter{page}{1}

\section{Introduction}

For $\lambda\ge0$, the (complex) Dunkl operators $D_{z}$ and $D_{\bar{z}}$ in the complex plane $\CC$,
as substitutes of
$\partial_z$ and
$\partial_{\bar{z}}$, are defined by
\begin{align*}
    D_{z}f(z)=\partial_z f+
    \lambda \frac{f(z)-f(\bar{z})}{z-\bar{z}},\qquad
    D_{\bar{z}}f(z)=\partial_{\bar{z}}f-
      \lambda\frac{f(z)-f(\bar{z})}{z-\bar{z}}.
\end{align*}
A $C^2$ function $f$ defined on the unit disk $\DD$ is said to be
$\lambda$-analytic, if $D_{\bar{z}}f=0$.
It was proved in \cite{LL1} that $f$ is $\lambda$-analytic in $\DD$ if and only if $f$ has the series representation
\begin{eqnarray}\label{anal-series-1-1}
f(z)=\sum_{n=0}^{\infty}c_{n}\phi_{n}^{\lambda}(z), \qquad |z|<1,
\end{eqnarray}
where
\begin{eqnarray*}
\phi_{n}^{\lambda}(z)=\epsilon_n\sum_{j=0}^{n}\frac{(\lambda)_{j}(\lambda+1)_{n-j}}
{j!(n-j)!}\bar{z}^{j}z^{n-j}, \qquad n\geq 0,
\end{eqnarray*}
and
\begin{align}\label{epsilon-bound-1}
\epsilon_n=\sqrt{n!/(2\lambda+1)_{n}}\asymp\sqrt{\Gamma(2\lambda+1)}(n+1)^{-\lambda}.
\end{align}
For details, see the next section. It is remarked that $\phi_0^{\lambda}(z)\equiv1$, and for $n\ge1$, $\phi_n^{0}(z)=z^n$.

The measure on $\DD$ associated with the operators $D_{z}$ and $D_{\bar{z}}$ is
\begin{eqnarray*}
d\sigma_{\lambda}(z)=c_{\lambda}|y|^{2\lambda}dxdy, \qquad z=x+iy,
\end{eqnarray*}
where $c_{\lambda}=\Gamma(\lambda+2)/\Gamma(\lambda+1/2)\Gamma(1/2)$ so that $\int_{\DD}d\sigma_{\lambda}(z)=1$.
For $0<p<\infty$, we denote by
$L^{p}(\DD;d\sigma_{\lambda})$, or simply by $L_{\lambda}^{p}(\DD)$,
the space of measurable functions $f$ on $\DD$ satisfying
$$
\|f\|_{L_{\lambda}^{p}(\DD)}:=\left(\int_{\DD}|f(z)|^{p}d\sigma_{\lambda}(z)\right)^{1/p}<\infty;
$$
and $L_{\lambda}^{\infty}(\DD)$, or simply $L^{\infty}(\DD)$, is the collection of all essentially
bounded measurable functions on $\DD$ with norm $\|f\|_{L^{\infty}(\DD)} ={\rm esssup}_{z\in\DD} |f(z)|$.
The associated Bergman space $A^{p}_{\lambda}(\DD)$, named the $\lambda$-Bergman
space, consists of those elements in $L_{\lambda}^{p}(\DD)$ that are
$\lambda$-analytic in $\DD$, and the norm of $f\in A^{p}_{\lambda}(\DD)$ is written as $\|f\|_{A_{\lambda}^{p}}$ instead of $\|f\|_{L_{\lambda}^{p}(\DD)}$.

Over the last decades, the theory of Bergman spaces has undergone a remarkable metamorphosis. Several breakthroughs in the early 1990s brought this field to a new stage, for example, H. Hedenmalm's construction of contractive canonical zero-divisors of functions in the $A^2$ space (see \cite{Hed1}), together with its generalization to the $A^p$ spaces for $0<p<\infty$ by P. Duren et al (see \cite{DKSS1}), characterizations of the interpolation and sampling sets of the Bergman space $A^2$ by K. Seip (see \cite{Se1,Se2}), and others.

In our previous paper \cite{LW1}, some fundamental aspects on the $\lambda$-Bergman spaces $A^{p}_{\lambda}(\DD)$ for $p_0\le p<\infty$ have been studied, where
$$
p_0=\frac{2\lambda}{2\lambda+1}.
$$
In the present paper we consider operators mapping the space $A^{p}_{\lambda}(\DD)$ for $p_0\le p\le1$ into a Banach space $X$. We shall prove a general theorem characterizing the boundedness of a operator $T$ from $A^{p}_{\lambda}(\DD)$ to $X$ by the behaviour of a single vector-valued $\lambda$-analytic function related to $T$. Such an idea comes from Blasco \cite{Bla0}, who dealt with the corresponding problem on the weighted Bergman spaces of usual analytic functions. The general theorem allows us to give several applications for particular operators acting on the $\lambda$-Bergman spaces $A^{p}_{\lambda}(\DD)$ with $p_0\le p\le1$.

Note that for $\lambda$-analytic functions, there are no analogue of Cauchy's theorem and also that of the Cauchy integral formula. Furthermore, roughly speaking, the product and the composition of $\lambda$-analytic functions are no longer $\lambda$-analytic. The lack of these tools and properties often causes difficulties, and makes us have to re-examine the processes in the classical theory and to find new approaches.

As in \cite{Bla0}, we shall work with the weighted analogues of the $\lambda$-Bergman spaces $A^{p}_{\lambda}(\DD)$, that are induced by a radial weight $w$ satisfying certain conditions. Precisely, for a function $w:\,[0,1)\mapsto(0,\infty)$, Lebesgue integrable over $[0,1)$, the weighted $\lambda$-Bergman space $A^{p}_{\lambda,w}(\DD)$ with $0<p<\infty$ is the collection of $\lambda$-analytic functions $f$ on $\DD$ for which
$$
\|f\|_{A_{\lambda,w}^{p}}:=\left(\int_{\DD}|f(z)|^{p}w(|z|)d\sigma_{\lambda}(z)\right)^{1/p}<\infty.
$$
Note that for $0<p<1$, $\|\cdot\|_{A_{\lambda,w}^{p}}$ is not a norm, however $\|f-g\|_{A_{\lambda,w}^{p}}^p$ defines a metric. Throughout the paper, the notation ${\mathcal{X}}\lesssim {\mathcal{Y}}$ or ${\mathcal{Y}}\gtrsim {\mathcal{X}}$ means that ${\mathcal{X}}\le c{\mathcal{Y}}$ for some positive constant $c$ independent of variables, functions, etc., and ${\mathcal{X}}\asymp {\mathcal{Y}}$ means that both ${\mathcal{X}}\lesssim {\mathcal{Y}}$ and ${\mathcal{Y}}\lesssim {\mathcal{X}}$ hold.

The weight functions $w$ to be considered satisfies the condition
\begin{align}\label{weight-condition-1-1}
\int_r^1w(s)ds\lesssim (1-r)w(r),\qquad r\in(0,1),
\end{align}
or, for some $q>0$,
\begin{align}\label{weight-condition-1-2}
\int_0^r\frac{w(s)}{(1-s)^q}ds\lesssim \frac{1}{(1-r)^q}\int_r^1w(s)ds,\qquad r\in(0,1).
\end{align}
We remark that the inequality (\ref{weight-condition-1-1}) is called the Dini condition in literatures, and a substitution of (\ref{weight-condition-1-2}) is defined by
\begin{align*}
\int_0^r\frac{w(s)}{(1-s)^q}ds\le \frac{w(r)}{(1-r)^{q-1}},\qquad r\in(0,1),
\end{align*}
that is called the $b_q$-condition. The class of radial weights satisfying one or both of them, or its modification or generalizations, proved to be of importance in the study of various problems. However in earlier researches on the associated weighted Bergman spaces of usual analytic functions, some monotonicity assumption on the weight $w$ is needed, see \cite{Bla0}-\cite{BloS1} for example. If $(1-r)w(r)$ is non-increasing, it follows from \cite[p. 446, Remark 1.1]{Bla0} that the $b_q$-condition and the Dini condition (\ref{weight-condition-1-1}) implies that $\int_r^1w(s)ds\asymp (1-r)w(r)$, i.e., the weight $w$ is regular; and hence, under such an assumption, the conditions (\ref{weight-condition-1-1}) and (\ref{weight-condition-1-2}) together are equivalent to the $b_q$-condition and the condition (\ref{weight-condition-1-1}). In this paper, we do not need to make any monotonicity assumption for weight functions.

In recent years, the operator theory and harmonic analysis on the weighted Bergman spaces of usual analytic functions associated to various classes of weights have received extraordinary attention, see, for example, \cite{PR1}-\cite{PRS1} and the references therein. In \cite{PR1} the class of rapidly increasing radial weights is introduced, so that the associated weighted Bergman spaces lie ``closer" to the Hardy spaces $H^p$ than any classical weighted Bergman spaces (associated to the power weights). A continuous function $w:\,[0,1)\mapsto(0,\infty)$ is said to be rapidly increasing if
\begin{align*}
\lim_{r\rightarrow1-}\frac{\int_r^1w(s)ds}{(1-r)w(r)}=\infty.
\end{align*}
By \cite[Lemma 1.2]{PR1} every rapidly increasing weight $w$ satisfies the condition (\ref{weight-condition-1-2}) for all $q>0$, and every regular weight $w$ satisfies the condition (\ref{weight-condition-1-2}) for appropriately large $q>0$.

For the theory of the classical weighted Bergman spaces, see \cite{DS1,HKZ1,Zhu1,Zhu2}.

The paper is organized as follows. In Section 2 we recall some knowledge on $\lambda$-analytic functions and $\lambda-$harmonic functions on the disk $\DD$, and in Section 3 some basic problems on the weighted $\lambda$-Bergman spaces $A^{p}_{\lambda,w}(\DD)$ are studied. In Section 4 we prove the main theorems of the paper, about the boundedness of operators from the weighted Bergman spaces $A^{p}_{\lambda,w}(\DD)$ into a Banach space, with $p_0\le p\le1$, for weight functions $w$ satisfying (\ref{weight-condition-1-1}) or/and (\ref{weight-condition-1-2}). Section 5 is devoted to the sharp estimates of the $p$-integral means $M_{\tau,p}(f;r)$ involving the parameter $\tau\ge0$, and Sections 6 and 7 to the sharp estimates of the $(\lambda,\alpha)$-Bergman kernel $K_{\lambda,\alpha}(z,\zeta)$.
In Section 8 we apply the main theorems to characterize the dual space of $A^{p}_{\lambda,w}(\DD)$ for the power weight $w=(1-|z|^2)^{\alpha-1}$ with $\alpha>0$, and also give a sufficient condition of Carleson type for boundedness of multiplication operators on the weighted Bergman space. Finally in Section 9, we obtain a necessary and sufficient conditions of sequence multipliers on the spaces $A^{p}_{\lambda,w}(\DD)$ for general weights $w$,

\section{The $\lambda$-analytic and $\lambda$-harmonic functions}

The associated measure on the circle $\partial\DD\simeq[-\pi,\pi]$ is
\begin{eqnarray*}
dm_{\lambda}(\theta)=\tilde{c}_{\lambda}|\sin\theta|^{2\lambda}d\theta, \ \
\ \ \ \ \tilde{c}_{\lambda}=c_{\lambda}/(2\lambda+2).
\end{eqnarray*}
For $0<p<\infty$, we denote by $L^p(\partial\DD;dm_{\lambda})$, or simply by $L_{\lambda}^{p}(\partial\DD)$,
the space of measurable functions $f$ on $\partial\DD$ satisfying
$\|f\|_{L_{\lambda}^{p}(\partial\DD)}:=\left(\int_{-\pi}^{\pi}|f(e^{i\theta})|^p
dm_{\lambda}(\theta)\right)^{1/p}<\infty$,
and for $p=\infty$, $L^{\infty}(\partial\DD;dm_{\lambda})=L^{\infty}(\partial\DD)$, the collection of all essentially
bounded measurable functions on $\partial\DD$ with norm $\|f\|_{L^{\infty}(\partial\DD)} ={\rm esssup}_{\theta} |f(e^{i\theta})|$. In addition, ${\frak B}_{\lambda}(\partial\DD)$ denotes
the space of Borel measures $d\nu$ on $\partial\DD$ for which
$\|d\nu\|_{{\frak B_{\lambda}(\partial\DD)}}
=\tilde{c}_{\lambda}\int_{-\pi}^{\pi}|\sin\theta|^{2\lambda}|d\nu(\theta)|$
are finite.

From \cite{Dun3,Dun4} it is known that
\begin{align}\label{basis-1}
\phi_{n}^{\lambda}(z)&=\epsilon_nr^{n}\left[\frac{n+2\lambda}{2\lambda}P_{n}^{\lambda}
(\cos\theta)+i\sin\theta P_{n-1}^{\lambda+1}(\cos\theta)\right], \qquad z=re^{i\theta},
\end{align}
where $P_{n}^{\lambda}(t)$, $n=0,1,\dots$, are the Gegenbauer
polynomials, and $P_{-1}^{\lambda+1}=0$, and the system
$\left\{1, \, \phi_{n}^{\lambda}(e^{i\theta}), \,
\overline{e^{i\theta}\phi_{n-1}^{\lambda}(e^{i\theta})}, \,\,
n\in\NN\right\}$
is an orthonormal basis of the Hilbert space
$L_{\lambda}^2(\partial\DD)$. Moreover, for $n\ge1$,
\begin{align*}
\bar{z}\overline{\phi_{n-1}^{\lambda}(z)}
=\epsilon_{n-1}r^{n}\left[\frac{n}{2\lambda}P_{n}^{\lambda}(\cos\theta)-
 i\sin\theta P_{n-1}^{\lambda+1}(\cos\theta)\right],\qquad z=re^{i\theta}.
\end{align*}

In what follows, we write $\phi_n(z)=\phi_n^{\lambda}(z)$ for simplicity.
According to \cite[(29)]{LL1} we have
\begin{eqnarray} \label{A-operator-4}
\epsilon_n\phi_n(z)
=2^{2\lambda+1}\tilde{c}_{\lambda-1/2}\int_{0}^{1}(sz+(1-s)\bar{z})^n(1-s)^{\lambda-1}s^{\lambda}ds
\end{eqnarray}
for $n\ge0$, and
\begin{eqnarray}\label{phi-bound-1}
|\phi_n(z)|\le\epsilon_n^{-1}|z|^n\asymp (n+1)^{\lambda}|z|^n/\sqrt{\Gamma(2\lambda+1)}.
\end{eqnarray}

The Laplacian associated with $D_{z}$ and $D_{\bar{z}}$, called the $\lambda$-Laplacian, is defined by
$\Delta_{\lambda}=4D_{z}D_{\bar{z}}=4D_{\bar{z}}D_{z}$,
which can be written explicitly as
\begin{eqnarray*}
\Delta_{\lambda} f=\frac{\partial^2 f}{ \partial x^2}+\frac
{\partial^2 f}{\partial y^2} +\frac{2\lambda}{y}\frac{\partial
f}{\partial y}-\frac{\lambda}{y^2} [f(z)-f(\bar{z})],\qquad z=x+iy.
\end{eqnarray*}
A $C^2$ function $f$ defined on $\DD$ is said to be $\lambda-$harmonic, if $\Delta_{\lambda}f=0$.

\begin{proposition} \label{anal-basis} {\rm(\cite[Proposition 2.2]{LL1})}
The functions $\phi_{n}(z)$ ($n\ge0$) are $\lambda$-analytic and $\bar{z}\overline{\phi_{n-1}}(z)$ ($n\ge1$) are $\lambda$-harmonic. Moreover $D_{z}\phi_{n}(z)=\sqrt{n(n+2\lambda)}\phi_{n-1}(z)$, $D_{z}(\bar{z}\overline{\phi_{n-1}}(z))=-\lambda\phi_{n-1}(z)$, and
\begin{align}\label{Tzphi-2}
D_{z}(z\phi_{n-1}(z))=(n+\lambda)\phi_{n-1}(z).
\end{align}
\end{proposition}

A finite linear combination of elements in the system $\{1,\,\phi_{1}(z),\cdots,\phi_{n}(z),\cdots\}$
is called a $\lambda$-analytic polynomial, and respectively, a finite linear combination of elements in
\begin{align}\label{harmonic-basis-1}
\left\{1, \, \phi_{n}^{\lambda}(z), \,
\overline{z\phi_{n-1}^{\lambda}(z)}, \,\,
n\in\NN\right\}
\end{align}
is called a $\lambda$-harmonic polynomial.

The $\lambda$-Cauchy kernel $C(z,w)$ and the $\lambda$-Poisson kernel $P(z,w)$, which reproduce, associated with the measure $dm_{\lambda}$ on the circle $\partial\DD$,
all $\lambda$-analytic polynomials
and $\lambda$-harmonic polynomials respectively, are given by (cf. \cite{Dun3})
\begin{align*}
C(z,w)&=\sum_{n=0}^{\infty}\phi_{n}(z)\overline{\phi_{n}(w)},\\
P(z,w)&=C(z,w)+\bar{z}w C(w,z).
\end{align*}

\begin{proposition} \label{Poisson-Cauchy} {\rm(cf. \cite{Dun3})}
The series in $C(z,w)$ is convergent absolutely for $zw\in\DD$ and uniformly for $zw$ in a compact subset of $\DD$. Moreover for $zw\in\DD$,
\begin{align}
C(z,w)&=\frac{1}{1-z\bar{w}}P_{0}(z,w),\label{Cauchy-kernel-2-2}\\
P(z,w)&=\frac{1-|z|^{2}|w|^{2}}{|1-z\bar{w}|^{2}}P_{0}(z,w),\nonumber
\end{align}
where
\begin{align} \label{Poi-0}
P_{0}(z,w)&=\frac{1}{|1-zw|^{2\lambda}}{}_2\!F_{1}\Big({\lambda,\lambda
\atop
  2\lambda+1};\frac{4({\rm Im} z)({\rm Im}
  w)}{|1-zw|^{2}}\Big)\nonumber\\
&=\frac{1}{|1-z\bar{w}|^{2\lambda}}{}_2\!F_{1}\Big({\lambda,\lambda+1
\atop
  2\lambda+1};-\frac{4({\rm Im} z)({\rm Im}
  w)}{|1-z\bar{w}|^{2}}\Big),
\end{align}
and ${}_2\!F_{1}[a,b;c;t]$ is the Gauss hypergeometric function.
\end{proposition}

A $\lambda$-harmonic function in $\DD$ has a series representation
in terms of the system (\ref{harmonic-basis-1}) as given in the following proposition.

\begin{proposition} \label{har-series} {\rm(\cite[Theorem 3.1]{LL1})}
If $f$ is a $\lambda$-harmonic function in $\DD$, then there are two
sequences $\{c_n\}$ and $\{\tilde{c}_n\}$ of complex numbers, such
that
\begin{eqnarray*}
f(z)=\sum_{n=0}^{\infty}
c_n \phi_{n}(z)+
  \sum_{n=1}^{\infty} \tilde{c}_n \bar{z}\overline{\phi_{n-1}(z)}
\end{eqnarray*}
for $z\in\DD$. Furthermore, the two sequences above are given by
\begin{align*}
c_n&=\lim_{r\rightarrow1-}\int_{-\pi}^{\pi}f(re^{i\varphi})\overline{\phi_{n}(e^{i\varphi})}
dm_{\lambda}(\varphi),\\
\tilde{c}_n&=\lim_{r\rightarrow1-}\int_{-\pi}^{\pi}f(re^{i\varphi})
e^{i\varphi}\phi_{n-1}(e^{i\varphi})dm_{\lambda}(\varphi).
\end{align*}
Moreover, for each real $\gamma$, the series
$\sum_{n\ge1}n^{\gamma}(|c_n|+|\tilde{c}_n|)r^n$ converges uniformly
for $r$ in every closed subset of $[0,1)$.
\end{proposition}

As stated in the first section, a $\lambda$-analytic
function $f$ on $\DD$ has a series representation as in (\ref{anal-series-1-1}); and moreover, such an $f$ could be characterized by a Cauchy-Riemann type system.

\begin{proposition} \label{anal-thm} {\rm(\cite[Theorem 3.7]{LL1})}
For a $C^2$ function $f=u+iv$ defined on $\DD$,
the following statements are equivalent:

{\rm (i)}  $f$ is $\lambda$-analytic;

{\rm (ii)} $u$ and $v$ satisfy the generalized Cauchy-Riemann
system
\begin{eqnarray*}
\left \{\partial_{x}u=D_{y}v,\atop  D_{y}u=-\partial_{x}v,\right.
\end{eqnarray*}
where
\begin{align*}
D_y u(x,y)=\partial_y u(x,y)+\frac{\lambda}{y}\left[u(x,y)-u(x,-y)\right].
\end{align*}

{\rm (iii)}  $f$ has the series representation
\begin{eqnarray} \label{anal-series-1}
f(z)=\sum_{n=0}^{\infty}c_{n}\phi_{n}(z), \qquad |z|<1,
\end{eqnarray}
where
$$
c_{n}=\lim_{r\rightarrow 1-}\int_{-\pi}^{\pi}f(re^{i\varphi})
\overline{\phi_{n}(e^{i\varphi})}dm_{\lambda}(\varphi).
$$
Moreover, for each real $\gamma$, the series
$\sum_{n\ge1}n^{\gamma}|c_n|r^n$ converges uniformly for $r$ in
every closed subset of $[0,1)$.
\end{proposition}

%

The $\lambda$-Poisson integral of $f\in L_{\lambda}^{1}(\partial\DD)$ is defined by
\begin{align} \label{Poisson-integral-1}
P(f;z)=\int_{-\pi}^{\pi}f(e^{i\varphi})P(z,e^{i\varphi})
dm_{\lambda}(\varphi),\qquad z=re^{i\theta}\in\DD,
\end{align}
and that of a measure $d\nu\in{\frak B}_{\lambda}(\partial\DD)$ by
\begin{eqnarray} \label{Poisson-Borel-1}
P(d\nu;z)=\tilde{c}_{\lambda}\int_{-\pi}^{\pi}P(z,e^{i\varphi})
|\sin\varphi|^{2\lambda}d\nu(\varphi),\qquad z=re^{i\theta}\in\DD.
\end{eqnarray}

\begin{proposition} \label{Poi-property-1} {\rm(\cite[Propositions 2.4 and 2.5]{LL1})}
Let $f\in L_{\lambda}^{1}(\partial\DD)$. Then

 {\rm (i)} the $\lambda$-Poisson integral $u(x,y)=P(f;z)$ ($z=x+iy$) is
$\lambda$-harmonic
in $\DD$;

{\rm (ii)} if write $P_r(f;\theta)=P(f;re^{i\theta})$, we have the ``semi-group" property
$$
P_s(P_rf;\theta)=P_{s\,r}(f;\theta), \qquad 0\le s,r<1;
$$

{\rm (iii)} $P(f;z) \ge 0$ if $f \ge 0$;

{\rm (iv)} for $f\in X=L_{\lambda}^{p}(\partial\DD)$ ($1\le p\le\infty$), or $C(\partial\DD)$,
$\|P_r(f;\cdot)\|_{X}\le \|f\|_{X}$;

{\rm (v)} for $f\in X=L_{\lambda}^{p}(\partial\DD)$ ($1\le p<\infty$), or
$C(\partial\DD)$, $\lim_{r\rightarrow1-}\|P_r(f;\cdot)-f\|_X=0$;

{\rm (vi)} the conclusions in (i)-(iii) are true also for $P_r(d\nu;\cdot)$ in place of $P_r(f;\cdot)$, if
$d\nu\in{\frak B}_{\lambda}(\partial\DD)$; moreover, $\|P_r(d\nu;\cdot)\|_{L_{\lambda}^{1}(\partial\DD)}\le \|d\nu\|_{\frak B_{\lambda}(\partial\DD)}$.
\end{proposition}

As usual, the $p$-means of a function $f$ defined on $\DD$, for $0<p<\infty$, are given by
\begin{eqnarray*}
M_p(f;r)=\left\{\int_{-\pi}^{\pi}|f(re^{i\theta})|^p\,dm_{\lambda}(\theta)\right\}^{1/p},\qquad 0\le r<1;
\end{eqnarray*}
and $M_{\infty}(f;r)=\sup_{\theta}|f(re^{i\theta})|$.
The $\lambda$-Hardy space $H_{\lambda}^p(\DD)$ is the collection of $\lambda$-analytic functions on $\DD$ satisfying
$$
\|f\|_{H_{\lambda}^p}:=\sup_{0\le r <1}M_p(f;r)<\infty.
$$
Obviously $H_{\lambda}^{\infty}(\DD)$ is identical with $A_{\lambda}^{\infty}(\DD)$.

The fundamental theory of the $\lambda$-Hardy spaces $H_{\lambda}^p(\DD)$ for $p\ge p_0$ was studied in \cite{LL1}.
The following theorem asserts the existence of boundary values of functions in $H_{\lambda}^p(\DD)$.

\begin{theorem} \label{Hardy-boundary-value-a} {\rm (\cite[Theorem 6.6]{LL1})}
Let $p\ge p_0$ and $f\in H_{\lambda}^p(\DD)$. Then for almost every $\theta\in[-\pi,\pi]$, $\lim f(r
e^{i\varphi})=f(e^{i\theta})$  exists as $r e^{i\varphi}$ approaches to the point
$e^{i\theta}$ nontangentially, and if $p_0<p<\infty$, then
\begin{align}\label{Hp-norm-convergence-1}
\lim_{r\rightarrow1-}\int_{-\pi}^{\pi}|f(r
e^{i\theta})-f(e^{i\theta})|^{p} dm_{\lambda}(\theta)=0
\end{align}
and
\begin{eqnarray}\label{Hp-norm-ineq-1}
\|f\|_{H^{p}_{\lambda}}\asymp\left(\int_{-\pi}^{\pi}|f(e^{i\theta})|^{p}
dm_{\lambda}(\theta)\right)^{1/p}.
\end{eqnarray}
%
\end{theorem}

%

C. Dunkl's work \cite{Dun3} built up a framework associated with the dihedral group
$G=D_k$ on the disk $\DD$. The study of the $\lambda$-analytic functions in this paper, and also in \cite{LL1} and \cite{LW1}, focus on the special case with $G=D_1$ having
the reflection $z\mapsto\overline{z}$ only, in order to find possibilities to develop a deep theory of associated function spaces. We note that C. Dunkl has a general theory named after him associated with reflection-invariance on the Euclidean spaces, see \cite{Dun1}, \cite{Dun2} and \cite{Dun4} for example. For the Hardy spaces in the upper half plane $\RR_+^2$ associated to the Dunkl operators $D_{z}$ and $D_{\bar{z}}$, see \cite{LL2}.


\section{Preliminaries to the weighted $\lambda$-Bergman spaces $A^{p}_{\lambda,w}(\DD)$}

We first consider the point-evaluation of functions in the weighted $\lambda$-Bergman spaces $A^{p}_{\lambda,w}(\DD)$.

\begin{lemma}\label{monotonicity-integral-mean-2-a-1} {\rm (\cite[Lemma 4.2]{LW1} or \cite{LL1})}
If $f$ is $\lambda$-harmonic in $\DD$ and $1\le p<\infty$, then its integral mean $r\mapsto M_p(f;r)$ is nondecreasing over $[0,1)$; and if $f$ is $\lambda$-analytic in $\DD$ and $p_0\le p<1$, then $M_p(f;r')\le2^{2/p-1}M_p(f;r)$ for $0\le r'<r<1$.
\end{lemma}

\begin{lemma} \label{Bergman-mean-estimate-w-a}
Let $p_0\le p<\infty$ and let $w$ be a nonzero weight function. If $f\in A^{p}_{\lambda,w}(\DD)$, then
$M_p(f;r)=o\left(\left(\int_r^1w(s)ds\right)^{-1/p}\right)$ as $r\rightarrow1-$, and
$$
M_p(f;r)\lesssim\left(\int_r^1w(s)ds\right)^{-1/p}\|f\|_{A_{\lambda,w}^{p}},\qquad 0\le r<1.
$$
\end{lemma}

\begin{proof} By Lemma \ref{monotonicity-integral-mean-2-a-1},
\begin{align*}
M_p(f;r)^p\int_r^1w(s)ds\lesssim\int_r^1M_p(f;s)^pw(s)ds,\qquad r\in[0,1).
\end{align*}
It remains to show that for $r\in[0,1)$,
\begin{align}\label{integral-mean-w-1}
\int_r^1M_p(f;s)^pw(s)ds\lesssim\int_r^1M_p(f;s)^pw(s)s^{2\lambda+1}ds.
\end{align}
Choose fixed $r_0\in(0,1)$ so that $\int_{r_0}^1w(s)ds>0$. For $r\in[r_0,1)$, (\ref{integral-mean-w-1}) is obvious, and for $r\in[0,r_0)$,
since
\begin{align*}
\int_r^{r_0}M_p(f;s)^pw(s)ds\lesssim M_p(f;r_0)^p\int_r^{r_0}w(s)ds
\lesssim \frac{\int_r^{r_0}w(s)ds}{\int_{r_0}^1w(s)ds}\int_{r_0}^1 M_p(f;s)^p w(s)ds,
\end{align*}
it follows that
\begin{align*}
\int_r^1M_p(f;s)^pw(s)ds\lesssim\left(\frac{\int_r^{r_0}w(s)ds}{\int_{r_0}^1w(s)ds}+1\right)\int_{r_0}^1M_p(f;s)^pw(s)s^{2\lambda+1}ds,
\end{align*}
which concludes (\ref{integral-mean-w-1}).
\end{proof}

\begin{lemma} \label{Hardy-point-evaluation-a-1} {\rm (\cite[Theorem 4.10]{LW1})}
Let $p\ge p_0$ and $f\in H_{\lambda}^p(\DD)$. Then
\begin{eqnarray*}
|f(z)|\lesssim\frac{(1-|z|)^{-1/p}}{|1-z^2|^{2\lambda/p}}\|f\|_{H_{\lambda}^{p}},\qquad z\in\DD.
\end{eqnarray*}
\end{lemma}

\begin{theorem} \label{Bergman-point-evaluation-w-a}
Let $p_0\le p<\infty$ and let $w$ be a nonzero weight function. If $f\in A_{\lambda,w}^p(\DD)$, then
\begin{eqnarray}\label{Bergman-point-evaluation-w-1}
|f(z)|\lesssim\frac{(1-|z|)^{-1/p}}{|1-z^2|^{2\lambda/p}}\left(\int_{r^{1/2}}^1w(s)ds\right)^{-1/p}\|f\|_{A_{\lambda,w}^{p}},\qquad z\in\DD.
\end{eqnarray}
\end{theorem}

\begin{proof}
For $0<\rho<1$ set $f_{\rho}(z)=f(\rho z)$. We apply Lemma \ref{Hardy-point-evaluation-a-1} to $f_{\rho}$ to get
\begin{eqnarray*}
|f(\rho se^{i\theta})|\lesssim\frac{(1-s)^{-1/p}}{|1-s^2e^{2i\theta}|^{2\lambda/p}}\|f_{\rho}\|_{H_{\lambda}^{p}},\qquad \rho,s\in[0,1),
\end{eqnarray*}
but from (\ref{Hp-norm-ineq-1}) and Lemma \ref{Bergman-mean-estimate-w-a},
$$
\|f_{\rho}\|_{H_{\lambda}^{p}}\lesssim M_p(f;\rho)\lesssim\left(\int_{\rho}^1w(s)ds\right)^{-1/p}\|f\|_{A_{\lambda,w}^{p}}.
$$
Combining the above two estimates and letting $\rho=s=r^{1/2}$ for $r\in[0,1)$, we have
\begin{eqnarray*}
|f(z)|\lesssim\frac{(1-r^{1/2})^{-1/p}}{|1-re^{2i\theta}|^{2\lambda/p}}\left(\int_{r^{1/2}}^1w(s)ds\right)^{-1/p}\|f\|_{A_{\lambda,w}^{p}}, \qquad z=re^{i\theta}\in\DD,
\end{eqnarray*}
that is identical with (\ref{Bergman-point-evaluation-w-1}) in view of the fact
\begin{align}\label{elementary-equality-2}
|1-z^2|\asymp 1-r+|\sin\theta|\asymp|1-re^{2i\theta}|,\qquad z=re^{i\theta}\in\DD.
\end{align}
\end{proof}

Next we turn to the completeness of the weighted $\lambda$-Bergman spaces $A^{p}_{\lambda,w}(\DD)$.

\begin{lemma}\label{analytic-converge-a-1} {\rm (\cite[Lemma 5.4]{LW1})}
Let $\{f_n\}$ be a sequence of $\lambda$-analytic functions on $\DD$. If $\{f_n\}$ converges uniformly on each compact subset of $\DD$, then its limit function is also $\lambda$-analytic in $\DD$.
\end{lemma}

\begin{theorem}\label{completeness-Bergman-w-a}
Let $p_0\le p<\infty$ and let $w$ be a nonzero radial weight function. Then the space $A^p_{\lambda,w}(\DD)$ is complete.
\end{theorem}

\begin{proof}
Let $\{f_n\}$ be a Cauchy sequence in $A^p_{\lambda,w}(\DD)$. For $r\in(0,1)$, it follows from Theorem \ref{Bergman-point-evaluation-w-a} that
\begin{eqnarray*}
|f_m(z)-f_n(z)|\lesssim\frac{1}{(1-r)^{(2\lambda+1)/p}}\left(\int_{r^{1/2}}^1w(s)ds\right)^{-1/p}\|f_m-f_n\|_{A_{\lambda,w}^{p}},\qquad z\in\overline{\DD}_r,
\end{eqnarray*}
which shows that $\{f_n\}$ converges to a function $f$ uniformly on each compact subset of $\DD$.
Lemma \ref{analytic-converge-a-1} asserts that $f$ is $\lambda$-analytic in $\DD$. The locally uniform convergence implies, for $r\in[0,1)$,
\begin{align*}
\int_{\DD_r}|f_n(z)-f(z)|^{p}w(|z|)d\sigma_{\lambda}(z) &=\lim_{m\rightarrow\infty}\int_{\DD_r}|f_n(z)-f_m(z)|^{p}w(|z|)d\sigma_{\lambda}(z)\\
&\le\liminf_{m\rightarrow\infty}\|f_n-f_m\|^p_{A_{\lambda,w}^{p}},
\end{align*}
so that $\|f_n-f\|_{A_{\lambda,w}^{p}}\le\liminf\limits_{m\rightarrow\infty}\|f_n-f_m\|_{A_{\lambda,w}^{p}}$. Hence $A^p_{\lambda,w}(\DD)$ is complete.
\end{proof}

The above theorem shows that for a radial weight function $w$, the weighted $\lambda$-Bergman space $A^{p}_{\lambda,w}(\DD)$ is a Banach space when $1\le p<\infty$, and a Fr\'{e}chet space with the quasi-norm $\|\cdot\|^p_{A_{\lambda,w}^{p}}$ when $p_0\le p<1$.

Finally we come to the density of $\lambda$-analytic polynomials in the weighted $\lambda$-Bergman spaces $f\in A^{p}_{\lambda,w}(\DD)$. For a $\lambda$-analytic function $f$ on $\DD$, we denote by $S_n:=S^f_n(z)$ the $n$th partial sum of the series (\ref{anal-series-1}).

\begin{lemma}\label{partial-sum-converge-a-1} {\rm (\cite[Lemma 5.1]{LW1})}
If $f(z)$ is $\lambda$-analytic in $\DD$, then for fixed $0<r_0<1$, $S_n(z)$ converges uniformly on $\overline{\DD}_{r_0}$ as $n\rightarrow\infty$, where $\DD_{r_0}=\{z:\, |z|<r_0\}$.
\end{lemma}

\begin{theorem} \label{Bergman-density-w-a}
Let $p_0\le p<\infty$ and let $w$ be a nonzero weight function. Then the set of $\lambda$-analytic polynomials is dense in the $\lambda$-Bergman space $A^{p}_{\lambda,w}(\DD)$; and in particular, the system
\begin{align*}
\left\{1, \,\, \phi_{n}^{\lambda}(z)/\|\phi_n\|_{A_{\lambda,w}^{2}}, \,\, n\in\NN\right\}
\end{align*}
is an orthonormal basis of the Hilbert space $A^{2}_{\lambda,w}(\DD)$, where
$$
\|\phi_n\|_{A_{\lambda,w}^{2}}=\left(2(\lambda+1)\int_0^1s^{2n+2\lambda+1}w(s)ds\right)^{1/2}.
$$
\end{theorem}

\begin{proof}
For $f\in A^{p}_{\lambda,w}(\DD)$, set $f_s(z)=f(sz)$ for $0\le s<1$.
We assert that $\|f_s-f\|_{A_{\lambda,w}^{p}}\rightarrow0$ as $s\rightarrow1-$, namely,
\begin{eqnarray}\label{Ap-norm-convergence-w-1}
\lim_{s\rightarrow1-}\int_0^1M_p(f_s-f;r)^p\,w(r)r^{2\lambda+1}dr=0.
\end{eqnarray}
Indeed, applying (\ref{Hp-norm-convergence-1}) yields $\lim_{s\rightarrow1-}M_p(f_s-f;r)=0$ for $0\le r<1$ (indeed, for all $p\ge p_0$ and fixed $r$ this follows from the uniform continuity of $f$ on $\overline{\DD}_r$), and for all $0\le r,s<1$, $M_p(f_s-f;r)\lesssim M_p(f;rs)+M_p(f;r)\lesssim M_p(f;r)$ by Lemma \ref{monotonicity-integral-mean-2-a-1}. Thus (\ref{Ap-norm-convergence-w-1}) follows immediately by Lebesgue's dominated convergence theorem. In addition, by Lemma \ref{partial-sum-converge-a-1}, for fixed $s\in(0,1)$ $\|(S_n)_s-f_s\|_{A^{p}_{\lambda,w}}\rightarrow0$ as $n\rightarrow\infty$. Thus the density in Theorem \ref{Bergman-density-w-a} is proved.
The last assertion in the theorem follows from the fact $\|\phi_n\|_{L_{\lambda}^{2}(\partial\DD)}=1$.
\end{proof}

\section{Operators mapping $A^{p}_{\lambda,w}(\DD)$ into a general Banach space}

\subsection{The main theorems and lemmas}

Let $(X,\|\cdot\|_X)$ be a Banach space and $T$ a linear operator mapping a $\lambda$-analytic polynomial into an element in $X$. In this section, we give a characterization of the boundedness of the operator $T$ from $A^{p}_{\lambda,w}(\DD)$ to $X$ by the behaviour of a single vector-valued $\lambda$-analytic function related to $T$.

We define the function $K_{\lambda,\alpha}(z,\zeta)$ for $|z\zeta|<1$, called the $(\lambda,\alpha)$-Bergman kernel, by
\begin{eqnarray}\label{Bergman-kernel-1}
K_{\lambda,\alpha}(z,\zeta)=\sum_{n=0}^{\infty}
\frac{\Gamma(\lambda+1)\Gamma(n+\lambda+\alpha+2)}{\Gamma(\lambda+\alpha+2)\Gamma(n+\lambda+1)}
\phi_{n}(z)\phi_{n}(\zeta).
\end{eqnarray}
It is obvious that $K_{\lambda,-1}(z,\overline{\zeta})$ is identical with the $\lambda$-Cauchy kernel $C(z,\zeta)$, and $K_{\lambda,0}(z,\overline{\zeta})$ with the $\lambda$-Bergman kernel $K_{\lambda}(z,\zeta)$ defined in \cite{LW1}.

Letting $x_n=T\phi_n$($\in X$) for $n=0,1,\cdots$, we define
\begin{eqnarray}\label{F-functin-1}
F_{\lambda,\alpha}(z)=\sum_{n=0}^{\infty}
\frac{\Gamma(\lambda+1)\Gamma(n+\lambda+\alpha+2)}{\Gamma(\lambda+\alpha+2)\Gamma(n+\lambda+1)}\,
x_n\phi_{n}(z),\qquad z\in\DD.
\end{eqnarray}

Our main theorems in this section are stated as follows.

\begin{theorem} \label{operator-boundedness-main-a}
Let $p_0\le p\le1$ and $0<q<\infty$, and let $w$ be a nonzero weight function satisfying the condition (\ref{weight-condition-1-1}). If $F_{\lambda,\alpha}$ is an $X$-valued $\lambda$-analytic function and satisfies
\begin{align}\label{F-function-norm-1}
\|F_{\lambda,\alpha}(z)\|_X^p\lesssim \frac{|1-z^{2}|^{2\lambda(1-p)}}{(1-|z|)^{(\alpha+2)p-1}}\int_{|z|^{\frac{1}{2}}}^1w(s)ds, \qquad z\in\DD,
\end{align}
where
\begin{align}\label{parameter-equation-1}
\alpha=\frac{q+2\lambda+1}{p}-2\lambda-2,
\end{align}
then $T$ can be extended to a bounded operator from $A^{p}_{\lambda,w}(\DD)$ to $X$.
\end{theorem}

\begin{theorem} \label{operator-boundedness-main-b}
Let $p_0\le p\le1$ and $0<q<\infty$, and let $w$ be a weight function satisfying the condition (\ref{weight-condition-1-2}). If the operator $T$ has a bounded extension from $A^{p}_{\lambda,w}(\DD)$ to $X$, then $F_{\lambda,\alpha}$ is an $X$-valued $\lambda$-analytic function and satisfies (\ref{F-function-norm-1}) with $\alpha$ given by (\ref{parameter-equation-1}).
\end{theorem}

Combining Theorems \ref{operator-boundedness-main-a} and \ref{operator-boundedness-main-b} we have the following theorem.

\begin{theorem} \label{operator-boundedness-main-c}
Let $p_0\le p\le1$ and $0<q<\infty$, and let $w$ be a nonzero weight function satisfying the conditions (\ref{weight-condition-1-1}) and (\ref{weight-condition-1-2}). Then $T$ can be extended to a bounded operator from $A^{p}_{\lambda,w}(\DD)$ to $X$ if and only if $F_{\lambda,\alpha}$ is an $X$-valued $\lambda$-analytic function and satisfies  (\ref{F-function-norm-1}) with $\alpha$ given by (\ref{parameter-equation-1}).
\end{theorem}

{\bf Remark.} We point out that the number $1/2$ on the right hand side of (\ref{F-function-norm-1}) could be replaced by arbitrary fixed number $\delta\in(0,1)$ for $p_0\le p<1$, and by $1$ for $p=1$; and in particular one may choose $\delta=p$ for all $p_0\le p\le1$. This could be seen from the proofs of these theorems below. We have to use such a choice in Section 9 for characterizing sequence multipliers on the weighted Bergman spaces $A^{p}_{\lambda,w}(\DD)$.

Lemmas \ref{p-means-estimate-a} and \ref{kernel-p-norm-a} below are the key in the proof of the above three theorems.

The $p$-integral means $M_{\tau,p}(f;r)$, involving the parameter $\tau\ge0$, of functions $f$ in the $\lambda$-Hardy space $H_{\lambda}^p(\DD)$ is defined by
\begin{align}\label{p-means-1}
M_{\tau,p}(f;r)=\left[\int_{-\pi}^{\pi}|1-r^{2}e^{2i\theta}|^{2\lambda\tau p} |f(re^{i\theta})|^{p}dm_{\lambda}(\theta)\right]^{1/p}, \qquad r\in[0,1).
\end{align}

\begin{lemma}\label{p-means-estimate-a}
Let $p_{0}\le p<1$, $\delta=p^{-1}-1$ and $0\le\tau\le\delta$. Then for $f\in H_{\lambda}^{p}(\DD)$,
\begin{eqnarray*}
M_{\tau,1}(f;r)\lesssim (1-r)^{2\lambda\tau-(2\lambda+1)\delta}\|f\|_{H_{\lambda}^{p}}, \qquad r\in[0,1).
\end{eqnarray*}
\end{lemma}

The proof of the lemma is postponed to the next section (see Corollary \ref{p-means-estimate-c}).

\begin{lemma}\label{kernel-p-mean-a}
For $0<p\le1$ and $\alpha>(2\lambda+1)(p^{-1}-1)-1$,
\begin{align*}
\int_{-\pi}^{\pi}|K_{\lambda,\alpha}(z,se^{i\varphi})|^{p}\,dm_{\lambda}(\varphi)
\lesssim \frac{|1-s^{2}z^{2}|^{2\lambda(1-p)}}{(1-s|z|)^{(2+\alpha)p-1}},\qquad z\in\DD,\,\,\,s\in[0,1).
\end{align*}
\end{lemma}

The proof of Lemma \ref{kernel-p-mean-a} is postponed to Sections six and seven (see Theorem \ref{kernel-p-mean-b}), but we now use it to prove the following lemma.

\begin{lemma}\label{kernel-p-norm-a}
Let $0<p\le1$ and $0<q<\infty$, and let $w$ be a weight function satisfying the condition (\ref{weight-condition-1-2}). Then
\begin{align}\label{kernel-p-norm-1}
\|{K_{\lambda,\alpha}}(z,\cdot)\|_{A_{\lambda,w}^{p}}^{p} \lesssim \frac{|1-z^{2}|^{2\lambda(1-p)}}{(1-|z|)^{(2+\alpha)p-1}}\int_{|z|^{\frac{1}{2}}}^1w(s)ds, \qquad z\in\DD,
\end{align}
where $\alpha$ is given by (\ref{parameter-equation-1}).
\end{lemma}

\begin{proof}
Note that for $q>0$, (\ref{parameter-equation-1}) implies $\alpha>(2\lambda+1)(p^{-1}-1)-1$, and then, by Lemma \ref{kernel-p-mean-a},
\begin{align*}
\|K_{\lambda,\alpha}(z,\cdot)\|_{A_{\lambda,w}^{p}}^{p}
\lesssim\int_{0}^{1}w(s)\frac{|1-s^{2}z^{2}|^{2\lambda(1-p)}}{(1-s|z|)^{(2+\alpha)p-1}}\,ds.
\end{align*}
Since $(2+\alpha)p-1=q+2\lambda(1-p)$ and
\begin{align}\label{elementary-equality-1}
|1-s^2z^2|\asymp 1-rs+|\sin\theta|\asymp 1-s+1-r+|\sin\theta|,\qquad z=re^{i\theta}\in\DD,\,\, s\in[0,1],
\end{align}
we have
\begin{align*}
\|K_{\lambda,\alpha}(z,\cdot)\|_{A_{\lambda,w}^{p}}^{p}
\lesssim\int_{0}^{1}w(s)\frac{(1-s+1-r+|\sin\theta|)^{2\lambda(1-p)}}{(1-sr)^{q+2\lambda(1-p)}}\,ds,\qquad z=re^{i\theta}\in\DD.
\end{align*}
Furthermore, dividing the above integral into two parts gives
\begin{align}\label{kernel-p-norm-2}
\|K_{\lambda,\alpha}(z,\cdot)\|_{A_{\lambda,\sigma}^{p}}^{p}
\lesssim &\int_{0}^{r^{1/2}}w(s)\frac{(1-s+|\sin\theta|)^{2\lambda(1-p)}}{(1-s)^{q+2\lambda(1-p)}}\,ds  \nonumber\\
&\qquad\qquad +\frac{(1-r+|\sin\theta|)^{2\lambda(1-p)}}{(1-r)^{q+2\lambda(1-p)}}\int_{r^{1/2}}^{1}w(s)\,ds.
\end{align}
Since $(1-s+|\sin\theta|)^{2\lambda(1-p)}\asymp(1-s)^{2\lambda(1-p)}+|\sin\theta|^{2\lambda(1-p)}$, the first integral above is dominated by
a multiple of
\begin{align*}
&\left(1+\frac{|\sin\theta|^{2\lambda(1-p)}}{(1-r)^{2\lambda(1-p)}}\right)\int_{0}^{r^{1/2}}\frac{w(s)}{(1-s)^{q}}\,ds.
\end{align*}
Substituting this into (\ref{kernel-p-norm-2}) and applying (\ref{weight-condition-1-2}) yields
\begin{align*}
\|K_{\lambda,\alpha}(z,\cdot)\|_{A_{\lambda,\sigma}^{p}}^{p}
\lesssim\frac{(1-r+|\sin\theta|)^{2\lambda(1-p)}}{(1-r)^{q+2\lambda(1-p)}}\int_{r^{1/2}}^{1}w(s)\,ds,\qquad z=re^{i\theta}\in\DD,
\end{align*}
that is identical with (\ref{kernel-p-norm-1}) on account of (\ref{elementary-equality-1}) with $s=1$.
\end{proof}

\subsection{Proofs of Theorems \ref{operator-boundedness-main-a} and \ref{operator-boundedness-main-b}}

{\it Proof of Theorem \ref{operator-boundedness-main-a}.}

By the density theorem, Theorem \ref{Bergman-density-w-a}, it suffices to show that $\|Tf\|_{X}\le C\|f\|_{A^{p}_{\lambda,w}}$ for all $\lambda$-analytic polynomials under the conditions (\ref{weight-condition-1-1}) and (\ref{F-function-norm-1}).

Assume that $f(z)=\sum_{n=0}^{m}c_{n}\phi_{n}(z)$. It follows that $Tf=\sum_{n=0}^{m}c_{n}x_n$, where $x_n=T\phi_n$ for $n=0,1,\cdots$. We now express $Tf$ into an integral in terms of $F_{\lambda,\alpha}(z)$.

Since, using the notation $\tau_n$ given by (\ref{tau-notation-1}),
$$
\frac{\Gamma(\lambda+1)\Gamma(\alpha+1)}{2\Gamma(\lambda+\alpha+2)}\frac{1}{\tau_n}=\frac{1}{2}\int_{0}^{1}(1-r)^{\alpha}r^{n+\lambda}dr=\int_{0}^{1}(1-r^{2})^{\alpha}r^{2n+2\lambda+1}dr,
$$
we have
\begin{align*}
Tf=\sum_{n=0}^{m}c_{n}x_{n}=\frac{2\Gamma(\lambda+\alpha+2)}{\Gamma(\lambda+1)\Gamma(\alpha+1)} \int _{0}^{1}(1-r^{2})^{\alpha}\left(\sum_{n=0}^{m}\tau_nc_{n}x_{n}r^{2n+2\lambda+1}\right)dr.
\end{align*}
But by the orthogonality,
\begin{align*}
\int^{\pi}_{-\pi}f(re^{i\theta})F_{\lambda,\alpha}(re^{-i\theta})\,dm_{\lambda}(\theta)=\sum_{n=0}^{m}\tau_nc_{n}x_{n}r^{2n},
\end{align*}
and hence
\begin{align*}
Tf=\frac{2\Gamma(\lambda+\alpha+2)}{\Gamma(\lambda+1)\Gamma(\alpha+1)} \int _{0}^{1}\int^{\pi}_{-\pi}f(re^{i\theta})F_{\lambda,\alpha}(re^{-i\theta})(1-r^{2})^{\alpha}r^{2\lambda+1}\,dm_{\lambda}(\theta)dr.
\end{align*}
Now applying the condition (\ref{F-function-norm-1}) we obtain
\begin{align}\label{operator-X-morm-1}
\|Tf \|_{X}\lesssim & \int _{0}^{1}\int^{\pi}_{-\pi}|f(re^{i\theta})|\|F_{\lambda,\alpha}(re^{-i\theta})\|_X(1-r^{2})^{\alpha}r^{2\lambda+1}\,dm_{\lambda}(\theta)dr \nonumber\\
\lesssim & \int _{0}^{1}\int^{\pi}_{-\pi}|f(re^{i\theta})|\frac{|1-r^{2}e^{-2i\theta}|^{2\lambda(p^{-1}-1)}}{(1-r)^{\alpha+2-p^{-1}}}
\left(\int_{r^{\frac{1}{2}}}^1w(s)ds\right)^{1/p} (1-r^{2})^{\alpha}r^{2\lambda+1}\,dm_{\lambda}(\theta)dr \nonumber\\
\lesssim & \int _{0}^{1}\frac{r^{2\lambda+1}}{(1-r)^{2-p^{-1}}}\left(\int_{r^{\frac{1}{2}}}^1w(s)ds\right)^{1/p}
 M_{p^{-1}-1,1}(f;r)dr,
\end{align}
where $M_{p^{-1}-1,1}(f;r)$ is given by (\ref{p-means-1}).

Again set $f_{\rho}(z)=f(\rho z)$ for $0\le\rho<1$. We apply Lemma \ref{p-means-estimate-a} to $f_{\rho}$ with $\tau=p^{-1}-1$ to get
\begin{eqnarray*}
M_{p^{-1}-1,1}(f_\rho;s)\lesssim (1-s)^{1-p^{-1}}\|f_{\rho}\|_{H_{\lambda}^{p}}, \qquad s\in[0,1).
\end{eqnarray*}
Choosing $\rho=s=r^{1/2}$ for $r\in[0,1)$ and appealing to (\ref{Hp-norm-ineq-1}) gives
\begin{align*}
\int^{\pi}_{-\pi}|1-re^{2i\theta}|^{2\lambda(p^{-1}-1)}|f(re^{i\theta})|\,dm_{\lambda}(\theta)
\lesssim (1-r)^{1-p^{-1}}M_p(f;r^{1/2}), \qquad r\in[0,1),
\end{align*}
which shows, on account of (\ref{elementary-equality-2}), $M_{p^{-1}-1,1}(f;r)\lesssim (1-r)^{1-p^{-1}}M_p(f;r^{1/2})$ for $r\in[0,1)$. Inserting this into (\ref{operator-X-morm-1}) yields
\begin{align*}
\|Tf \|_{X}
&\lesssim \int _{0}^{1}\frac{r^{2\lambda+1}}{1-r}\left(\int_{r^{1/2}}^1w(s)ds\right)^{1/p}M_p(f;r^{1/2})dr\\
&\lesssim \int _{0}^{1}\frac{t^{2\lambda+1}}{1-t}\left(\int_{t}^1w(s)ds\right)^{1/p}M_p(f;t)dt.
\end{align*}
If we set
\begin{eqnarray*}
\Lambda(f)=\sup_{0\le t<1 } M_p(f;t)^p\int_{t}^1w(s)ds,
 \end{eqnarray*}
then
 \begin{align}\label{operator-X-morm-2}
\|Tf \|_{X}
\lesssim \Lambda(f)^{p^{-1}-1}\int _{0}^{1}\frac{t^{2\lambda+1}}{1-t}\left(\int_{t}^1w(s)ds\right) M_p(f;t)^p\,dt.
\end{align}

By Lemma \ref{monotonicity-integral-mean-2-a-1} and (\ref{integral-mean-w-1}), one has
\begin{eqnarray*}
\Lambda(f)\lesssim\sup_{0\le t<1 } \int_{t}^1 M_p(f;s)^p w(s)ds
\lesssim\sup_{0\le t<1 } \int_{t}^1 M_p(f;s)^p w(s)s^{2\lambda+1}ds
\lesssim\|f\|_{A_{\lambda,w}^{p}}^p;
 \end{eqnarray*}
and using the condition (\ref{weight-condition-1-1}),
\begin{align*}
\int _{0}^{1}\frac{t^{2\lambda+1}}{1-t}\left(\int_{t}^1w(s)ds\right) M_p(f;t)^p\,dt
\lesssim\int _{0}^{1} M_p(f;t)^p\,w(t)t^{2\lambda+1}dt\lesssim\|f\|_{A_{\lambda,w}^{p}}^p.
\end{align*}
Substituting the two estimates above into (\ref{operator-X-morm-2}) gives
$\|Tf \|_{X}\lesssim \|f\|_{A_{\lambda,w}^{p}}^{1-p}\|f\|_{A_{\lambda,w}^{p}}^p=\|f\|_{A_{\lambda,w}^{p}}$.
In view of the density of $\lambda$-analytic polynomials in $A^{p}_{\lambda,w}(\DD)$, the proof of Theorem \ref{operator-boundedness-main-a} is finished.

{\it Proof of Theorem \ref{operator-boundedness-main-b}.}

Assume that the operator $T$ has a bounded extension from $A^{p}_{\lambda,w}(\DD)$ to $X$. For $M>N\ge0$, we have
\begin{align}\label{Bergman-kernel-X-morm-1}
\left\|\sum_{n=N}^{M}\tau_n x_n\phi_{n}(z)\right\|_{X}
&=\left\|T\left(\sum_{n=N}^{M}\tau_n \phi_n(z)\phi_{n}(\cdot)\right)\right\|_{X} \nonumber\\
&\le\|T\|\left\|\sum_{n=N}^{M}\tau_n \phi_n(z)\phi_{n}(\cdot)\right\|_{A^{p}_{\lambda,w}}, \qquad z\in\DD,
\end{align}
where
\begin{align}\label{tau-notation-1}
\tau_n=\frac{\Gamma(\lambda+1)\Gamma(n+\lambda+\alpha+2)}{\Gamma(\lambda+\alpha+2)\Gamma(n+\lambda+1)}
\asymp (n+1)^{\alpha+1}.
\end{align}
However from (\ref{phi-bound-1}), $\|\phi_{n}\|_{A^{p}_{\lambda,w}}^p\lesssim(n+1)^{\lambda p}$, and so
\begin{align*}
\left\|\sum_{n=N}^{M}\tau_n \phi_n(z)\phi_{n}(\cdot)\right\|_{A^{p}_{\lambda,w}}^p
&\le\sum_{n=N}^{M}\tau_n^p |\phi_n(z)|^p\|\phi_{n}\|_{A^{p}_{\lambda,w}}^p \\
&\lesssim \sum_{n=N}^{M}(n+1)^{(\alpha+2\lambda+1)p}|z|^{np},\qquad x\in\DD.
\end{align*}
In view of completeness of the space $A^{p}_{\lambda,w}(\DD)$,
this shows that for $|z|<1$, the series $\zeta\mapsto\sum_{n=0}^{\infty}\tau_n \phi_n(z)\phi_{n}(\zeta)$ converges in the space $A^{p}_{\lambda,w}(\DD)$, namely, for $|z|<1$ the $(\lambda,\alpha)$-Bergman kernel $K_{\lambda,\alpha}(z,\cdot)$ defined in (\ref{Bergman-kernel-1}) is in $A^{p}_{\lambda,w}(\DD)$; and moreover, in conjunction with (\ref{Bergman-kernel-X-morm-1}), it is concluded that the series $\sum_{n=0}^{\infty}\tau_n x_n\phi_{n}(z)$ defines an $X$-valued $\lambda$-analytic function of $z\in\DD$, denoted by $F_{\lambda,\alpha}(z)$ as in (\ref{F-functin-1}). Now letting $N=0$ and $M\rightarrow\infty$ in (\ref{Bergman-kernel-X-morm-1}), we obtain
\begin{eqnarray*}
\left\|F_{\lambda,\alpha}(z)\right\|_{X}
\le\|T\|\left\|K_{\lambda,\alpha}(z,\cdot)\right\|_{A^{p}_{\lambda,w}},\qquad z\in\DD.
\end{eqnarray*}
Direct application of Lemma \ref{kernel-p-norm-a} proves that $F_{\lambda,\alpha}(z)$ satisfies (\ref{F-function-norm-1}).
The proof of Theorem \ref{operator-boundedness-main-b} is completed.

\section{Estimation of the $p$-integral means $M_{\tau,p}(f;r)$}

We define the $p$-integral means $M_{\tau,p}(f;r)$ involving the parameter $\tau\ge0$ by
\begin{align*}
M_{\tau,p}(f;r)&=\left[\int_{-\pi}^{\pi}|1-r^{2}e^{2i\theta}|^{2\lambda\tau p} |f(re^{i\theta})|^{p}dm_{\lambda}(\theta)\right]^{1/p}, \qquad 0<p<\infty,\\
M_{\tau,\infty}(f;r)&=\sup\limits_{\theta}|1-r^{2}e^{2i\theta}|^{2\lambda\tau}|f(re^{i\theta})|.
\end{align*}

We are going to give some estimates of the $p$-integral means $M_{\tau,p}(f;r)$ of functions $f$ in the $\lambda$-Hardy space $H_{\lambda}^p(\DD)$. We first consider the case when $1\le p\le\infty$. In this case we shall work with the $\lambda$-Poisson integral $P(f;z)$ of $f\in L_{\lambda}^{1}(\partial\DD)$. Similarly, for $\tau\ge0$ we define
\begin{align*}
M_{\tau,p}(f;r)&=\left[\int_{-\pi}^{\pi}|1-r^{2}e^{2i\theta}|^{2\lambda\tau p} |P(f;re^{i\theta})|^p dm_{\lambda}(\theta)\right]^{1/p}, \qquad 1\le p<\infty,\\
M_{\tau,\infty}(f;r)&=\sup\limits_{\theta}|1-r^{2}e^{2i\theta}|^{2\lambda\tau}|P(f;re^{i\theta})|,
\end{align*}
and in the same way, $M_{\tau,p}(d\nu;r)$ for $d\nu\in {\frak B}_{\lambda}(\partial\DD)$.

\begin{lemma} \label{Poisson-kernel-estimate-a-1} {\rm (\cite[Corollary 4.3]{LL1})}
For $\theta,\varphi\in[-\pi,\pi]$ and $r\in[0,1)$, we have
\begin{eqnarray*}
P(re^{i\theta},e^{i\varphi})
\lesssim\frac{(1-r)(1-r+|\sin(\theta-\varphi)/2|)^{-2}}{(1-r+|\sin\theta|+|\sin(\theta-\varphi)/2|)^{2\lambda}}
\ln\Big(\frac{|\sin(\theta-\varphi)/2|}{1-r}+2\Big).
\end{eqnarray*}
\end{lemma}

\begin{theorem} \label{p-means-estimate-b}
Let $1\le p\le\ell\le\infty$, $\delta=p^{-1}-\ell^{-1}$ and $0\le\tau\le\delta$, and let $f\in L_{\lambda}^{p}(\partial\DD)$. Then

{\rm (i)} \ for $1\le p\le\ell\le+\infty$, $M_{\tau,\ell}(f;r)\lesssim (1-r)^{2\lambda\tau-(2\lambda+1)\delta}\|f\|_{L_{\lambda}^{p}(\partial\DD)}$ for $r\in[0,1)$;

{\rm (ii)} \ for $1\le p<\ell\le+\infty$,
$M_{\tau,\ell}(f;r)=o\left((1-r)^{2\lambda\tau-(2\lambda+1)\delta}\right)$ as $r\rightarrow1-$;

{\rm (iii)} \ for $1<p<\ell\le+\infty$, $p\le k<+\infty$ and $0\le\tau<\delta$,
\begin{eqnarray} \label{p-means-estimate-2}
\left(\int_0^1(1-r)^{k\delta(2\lambda+1)-2\lambda\tau k-1}M_{\tau,\ell}(f;r)^kdr\right)^{1/k}
\lesssim\|f\|_{L_{\lambda}^{p}(\partial\DD)};
\end{eqnarray}

{\rm (iv)} for $d\nu\in {\frak B}_{\lambda}(\partial\DD)$, $1\le\ell\le+\infty$, $\delta=1-\ell^{-1}$ and $0\le\tau\le\delta$,
$$
M_{\tau,\ell}(d\nu;r)\lesssim(1-r)^{2\lambda\tau-\delta(1+2\lambda)}\|d\nu\|_{{\frak
B}_{\lambda}(\partial\DD)},\qquad r\in[0,1).
$$
\end{theorem}

\begin{proof}
Choose $q$ so that $q^{-1}=1-(p^{-1}-\ell^{-1})$. Then $1\leq q \leq p'$, or equivalently $1\leq p \leq q' \leq\infty$. For $r\in[0,1)$, define the operator $U_r$ by
\begin{eqnarray}\label{definition-U-1}
U_rf(\theta)=|1-r^{2}e^{2i\theta}|^{2\lambda\tau}P(f;re^{i\theta}), \qquad \theta\in[-\pi,\pi).
\end{eqnarray}

We shall prove
\begin{eqnarray}\label{Poisson-norm-1}
\|U_rf\|_{L_{\lambda}^{\ell}(\partial\DD)} \leq \|f\|_{L_{\lambda}^p(\partial\DD)} \sup_{\theta,\varphi} \mid 1-r^{2}e^{2i\theta}\mid^{2\lambda \tau} \left[P(re^{i\theta},e^{i\varphi})\right]^{1-q^{-1}}.
\end{eqnarray}
At first, for $p_1=q'$ and $\ell_1=+\infty $, from (\ref{Poisson-integral-1}), (\ref{definition-U-1}) and H\"older's inequality one has
\begin{align*}
\|U_rf\|_{L^{\infty}(\partial\DD)}
\leq  \|f\|_{L_{\lambda}^{q'}(\partial\DD)}\sup_{\theta} |1-r^{2}e^{2i\theta}|^{2\lambda\tau} \|P(re^{i\theta},\cdot)\|_{L_{\lambda}^q(\partial\DD)}.
\end{align*}
Writing
\begin{align*}
\left[P(re^{i\theta},e^{i\varphi})\right]^q=\left[P(re^{i\theta},e^{i\varphi})\right]^{q-1}P(re^{i\theta},e^{i\varphi})
\end{align*}
and noting that $\|P(re^{i\theta},\cdot)\|_{L_{\lambda}^1(\partial\DD)}=1$, we further obtain
\begin{align}\label{Poisson-norm-2}
\|U_rf\|_{L^{\infty}(\partial\DD)}
\leq & \|f\|_{L_{\lambda}^{q'}(\partial\DD)} \sup_{\theta,\varphi } |1-r^{2}e^{2i\theta}|^{2\lambda \tau} \left[P(re^{i\theta},e^{i\varphi})\right]^{1-q^{-1}}.
\end{align}

Secondly, for $p_{2}=1$ and $\ell_{2}=q$, it follows from (\ref{Poisson-integral-1}) and H\"older's inequality that
\begin{align*}
|P(f,re^{i\theta})|\leq \|f\|_{L_{\lambda}^1(\partial\DD)}^{\frac{1}{q'}} \||f|^{\frac{1}{q}}P(re^{i\theta },\cdot )\|_{L_{\lambda}^q(\partial\DD)},
\end{align*}
and then, from (\ref{definition-U-1}) and Fubini's theorem,
\begin{align*}
\|U_rf\|_{L_{\lambda}^q(\partial\DD)} &\le \|f\|_{L_{\lambda}^1(\partial\DD)}^{\frac{1}{q'}}
\left(\int_{-\pi}^{\pi}|1-r^{2}e^{2i\theta}|^{2\lambda\tau q}\||f|^{\frac{1}{q}}P(re^{i\theta},\cdot)\|_{L_{\lambda}^q(\partial\DD)}^{q} dm_{\lambda}(\theta)\right)^{\frac{1}{q}}\\
&= \|f\|_{L_{\lambda}^1(\partial\DD)}^{\frac{1}{q'}} \left(\int _{-\pi}^{\pi}|f(e^{i\varphi})| \int_{-\pi}^{\pi}
 |1-r^{2}e^{2i\theta}|^{2\lambda\tau q} \left[P(re^{i\theta},e^{i\varphi})\right]^{q} dm_{\lambda}(\theta)dm_{\lambda}(\varphi)\right)^{\frac{1}{q}}.
\end{align*}
Again using $\|P(r\cdot,e^{i\varphi})\|_{L_{\lambda}^1(\partial\DD)}=1$, the inner integral above is dominated by
\begin{align}\label{sup-Poisson-kernel-1}
\sup_{\theta,\varphi} |1-r^{2}e^{2i\theta}|^{2\lambda\tau q} \left[P(re^{i\theta},e^{i\varphi})\right]^{q-1},
\end{align}
and hence
\begin{align}\label{Poisson-norm-3}
\|U_rf\|_{L_{\lambda}^q(\partial\DD)}\le \|f\|_{L_{\lambda}^1(\partial\DD)} \sup_{\theta,\varphi} |1-r^{2}e^{2i\theta}|^{2\lambda\tau} \left[P(re^{i\theta},e^{i\varphi})\right]^{1-q^{-1}}.
\end{align}

Finally (\ref{Poisson-norm-1}) is yielded by applying the Riesz-Th\"orin interpolation theorem to (\ref{Poisson-norm-2}) and (\ref{Poisson-norm-3}).

Now by Lemma \ref{Poisson-kernel-estimate-a-1} and on account of (\ref{elementary-equality-1}) with $s=1$, we have
\begin{eqnarray*}
P(re^{i\theta},e^{i\varphi})
\lesssim\frac{(1-r)^{-1}}{(1-r+|\sin\theta|)^{2\lambda}}
\asymp\frac{(1-r)^{-1}}{|1-r^{2}e^{2i\theta}|^{2\lambda}},\qquad \theta,\varphi\in[-\pi,\pi],\,\,\, r\in[0,1),
\end{eqnarray*}
and so
\begin{align}\label{sup-Poisson-kernel-2}
|1-r^{2}e^{2i\theta}|^{2\lambda\tau} \left[P(re^{i\theta},e^{i\varphi})\right]^{\delta}
\lesssim\frac{(1-r)^{-\delta}}{|1-r^{2}e^{2i\theta}|^{2\lambda(\delta-\tau)}}
\lesssim(1-r)^{2\lambda\tau-(2\lambda+1)\delta}.
\end{align}
Inserting this into (\ref{Poisson-norm-1}) and noting $1-q^{-1}=p^{-1}-\ell^{-1}=\delta$, part (i) of the theorem is proved.

Next we come to showing part (ii). For $\epsilon>0$, we write
$f=f_1+f_2$, where $f_1$ is bounded and $\|f_2\|_{L_{\lambda}^p(\partial\DD)}<\epsilon$. Since for $1\le p<\ell\le\infty$ and $0\le\tau\le\delta$, one has $2\lambda\tau-(2\lambda+1)\delta\le-\delta<0$,
there exists $\eta>0$ such that $\|f_1\|_{L_{\lambda}^{\ell}(\partial\DD)}\le(1-r)^{2\lambda\tau-(2\lambda+1)\delta}\epsilon$ whenever
$1-r<\eta$. Note that $M_{\tau,\ell}(f;r)\le M_{\tau,\ell}(f_1;r)+M_{\tau,\ell}(f_2;r)$, but by the definition of $M_{\tau,\ell}(f_1;r)$ and Proposition \ref{Poi-property-1}(iv),
\begin{align*}
M_{\tau,\ell}(f_1;r)\lesssim
\left[\int_{-\pi}^{\pi}|P(f_1;re^{i\theta})|^{\ell} dm_{\lambda}(\theta)\right]^{1/\ell}
\le\|f_1\|_{L_{\lambda}^{\ell}(\partial\DD)}
\le(1-r)^{2\lambda\tau-(2\lambda+1)\delta}\epsilon.
\end{align*}
Moreover by part (i) just proved,
$$
M_{\tau,\ell}(f_2;r)\lesssim(1-r)^{2\lambda\tau-(2\lambda+1)\delta}\|f_2\|_{L_{\lambda}^{p}(\partial\DD)}
<(1-r)^{2\lambda\tau-(2\lambda+1)\delta}\epsilon.
$$
Thus it is concluded that $M_{\tau,\ell}(f;r)\lesssim(1-r)^{2\lambda\tau-(2\lambda+1)\delta}\epsilon$. Part (ii) is proved.

In order to prove part (iii), we fix $\ell\in(1,\infty)$ and $\tau\ge0$ satisfying $\ell^{-1}+\tau<1$. Consider the operator $T$ defined by, for $f\in L_{\lambda}^{1}(\partial\DD)$,
$$
Tf(r)=(1-r)^{-2\lambda\tau-(1+2\lambda)/\ell}M_{\tau,\ell}(f;r),\qquad r\in[0,1).
$$
It is clear that $T$ is quasi-linear. We are going to prove that for $1\le p\le 1/(\tau+\ell^{-1})$, the operator $T$ is of weak type $(p,p)$ from $L_{\lambda}^{p}(\partial\DD)$ to the measure space $\left([0,1),d\nu\right)$ with $d\nu=(1-r)^{2\lambda}dr$.

Since $1\le p\le 1/(\tau+\ell^{-1})$ implies $0\le\tau\le p^{-1}-\ell^{-1}$, from part (i) it follows that
$$
Tf(r)\le c(1-r)^{-(1+2\lambda)/p}\|f\|_{L_{\lambda}^{p}(\partial\DD)},\qquad r\in[0,1).
$$
Now for $s>0$, $\{r:\,\, Tf(r)>s\}\subseteq[1-s',1)$, where $s'=\left(c\|f\|_{L_{\lambda}^{p}(\partial\DD)}/s\right)^{p/(1+2\lambda)}$, and then
\begin{eqnarray*}
\nu\left(\{r:\,\, Tf(r)>s\}\right)\le\int_{1-s'}^1(1-r)^{2\lambda}dr=\frac{1}{1+2\lambda}\left(\frac{c}{s}\|f\|_{L_{\lambda}^{p}(\partial\DD)}\right)^p.
\end{eqnarray*}
Hence $T$ is of weak type $(p,p)$ for all $1\le p\le 1/(\tau+\ell^{-1})$.

By Marcinkiewicz's interpolation theorem, $T$ is of strong type $(p,p)$ for all $1<p<1/(\tau+\ell^{-1})$, i.e., $\int_0^1Tf(r)^pd\nu(r)\lesssim\|f\|_{L_{\lambda}^{p}(\partial\DD)}^p$, or equivalently,
\begin{eqnarray}\label{p-means-estimate-3}
\int_0^1(1-r)^{p\delta(2\lambda+1)-2\lambda\tau p-1}M_{\tau,\ell}(f;r)^pdr
\lesssim\|f\|_{L_{\lambda}^{p}(\partial\DD)}^p.
\end{eqnarray}
This is the case for $k=p$ in (\ref{p-means-estimate-2}).

For $p<k<\infty$, we write $(1-r)^{k\delta(2\lambda+1)-2\lambda\tau k-1}M_{\tau,\ell}(f;r)^k$ into the form
\begin{eqnarray*}
\left[(1-r)^{\delta(2\lambda+1)-2\lambda\tau}M_{\tau,\ell}(f;r)\right]^{k-p}(1-r)^{p\delta(2\lambda+1)-2\lambda\tau p-1}M_{\tau,\ell}(f;r)^p.
\end{eqnarray*}
By part (i), the first factor above is dominated by a multiple of $\|f\|_{L_{\lambda}^{p}(\partial\DD)}^{k-p}$, and so
\begin{align*}
\int_0^1(1-r)^{k\delta(2\lambda+1)-2\lambda\tau k-1}M_{\tau,\ell}(f;r)^kdr
\lesssim \|f\|_{L_{\lambda}^{p}(\partial\DD)}^{k-p} \int_0^1(1-r)^{p\delta(2\lambda+1)-2\lambda\tau p-1}M_{\tau,\ell}(f;r)^pdr.
\end{align*}
In view of (\ref{p-means-estimate-3}), this is dominated by a multiple of $\|f\|_{L_{\lambda}^{p}(\partial\DD)}^k$, and part (iii) is proved.

Part (iv) can be verified in a similar way to part (i). By
H\"older's inequality, from (\ref{Poisson-Borel-1}) we have
$$
|P(d\nu;re^{i\theta})|^{\ell}\le\|d\nu\|_{{\frak
B}_{\lambda}(\partial\DD)}^{\ell/\ell'}c_{\lambda}\int_{-\pi}^{\pi}|P(re^{i\theta},e^{i\varphi})|^{\ell}
|\sin\varphi|^{2\lambda}\left|d\nu(\varphi)\right|.
$$
Taking integration over $[-\pi,\pi]$ with respect to $|1-r^{2}e^{2i\theta}|^{2\lambda\tau\ell}dm_{\lambda}(\theta)$ yields
$$
M_{\tau,\ell}(d\nu;r)^{\ell} \le\|d\nu\|_{{\frak
B}_{\lambda}(\partial\DD)}^{\ell/\ell'} c_{\lambda}\int_{-\pi}^{\pi} \left[\int_{-\pi}^{\pi}|1-r^{2}e^{2i\theta}|^{2\lambda\tau\ell}|P(re^{i\theta},e^{i\varphi})|^{\ell} dm_{\lambda}(\theta)\right]
|\sin\varphi|^{2\lambda}\left|d\nu(\varphi)\right|.
$$
The inner integral above is dominated by the quantity (\ref{sup-Poisson-kernel-1}) with $\ell$ instead of $q$, and so
\begin{align*}
M_{\tau,\ell}(d\nu;r)\le \|d\nu\|_{{\frak
B}_{\lambda}(\partial\DD)} \sup_{\theta,\varphi} |1-r^{2}e^{2i\theta}|^{2\lambda\tau} \left[P(re^{i\theta},e^{i\varphi})\right]^{1-\ell^{-1}}.
\end{align*}
Finally appealing to (\ref{sup-Poisson-kernel-2}) with $\delta=1-1/\ell$ proves part (iv) immediately.
\end{proof}

We now turn to estimates of the $p$-integral means $M_{\tau,p}(f;r)$ of $H^p_{\lambda}(\DD)$ functions.
We shall need the $\lambda$-harmonic majorizations of functions in $H_{\lambda}^p(\DD)$.

\begin{lemma} \label{majorization-a-1} {\rm (\cite[Theorem 6.3]{LL1}, $\lambda$-harmonic majorization)}
Let $p\ge p_0$ and $f\in H_{\lambda}^p(\DD)$.

  {\rm (i)} \  If $p >p_0$, then there exists a nonnegative function $g(\theta)\in
L_{\lambda}^{p/p_0}(\partial\DD)$ such that for $z\in\DD$,
$|f(z)|^{p_0}\leq P(g;z)$,
and
$\|f\|^{p_0}_{H^{p}_{\lambda}}\le\|g\|_{L_{\lambda}^{p/p_0}(\partial\DD)}\le2^{2-p_0}\|f\|^{p_0}_{H^{p}_{\lambda}}$.

{\rm (ii)} \ If $p=p_0,$ there exists a finite positive measure
  $d\nu\in {\frak B}_{\lambda}(\partial\DD)$ such that for $z\in\DD$,
$|f(z)|^{p_0}\le P(d\nu;z)$,
and
$\|f\|^{p_0}_{H^{p_0}_{\lambda}}\le \|d\nu\|_{{\frak B}_{\lambda}(\partial\DD)}
\le2^{2-p_0}\|f\|^{p_0}_{H^{p_0}_{\lambda}}$.
\end{lemma}

\begin{theorem} \label{p-means-estimate-c}
Let $p_0\le p\le\ell\le\infty$, $\delta=p^{-1}-\ell^{-1}$ and $0\le\tau\le\delta$, and let $f\in H_{\lambda}^{p}(\DD)$. Then

{\rm (i)} \ for $p_0\le p\le\ell\le+\infty$, $M_{\tau,\ell}(f;r)\lesssim (1-r)^{2\lambda\tau-(2\lambda+1)\delta}\|f\|_{H_{\lambda}^{p}}$ for $r\in[0,1)$;

{\rm (ii)} \ for $p_0<p<\ell\le+\infty$,
$M_{\tau,\ell}(f;r)=o\left((1-r)^{2\lambda\tau-(2\lambda+1)\delta}\right)$ as $r\rightarrow1-$;

{\rm (iii)} \ for $p_0<p<\ell\le+\infty$, $p\le k<+\infty$ and $0\le\tau<\delta$,
\begin{eqnarray} \label{p-means-estimate-4}
\left(\int_0^1(1-r)^{k\delta(2\lambda+1)-2\lambda\tau k-1}M_{\tau,\ell}(f;r)^kdr\right)^{1/k}
\lesssim\|f\|_{H_{\lambda}^{p}};
\end{eqnarray}
\end{theorem}

\begin{proof}
For $\ell\ge p>p_0$ and $f\in H_{\lambda}^{p}(\DD)$, by Lemma \ref{majorization-a-1}(i) we have $|f(z)|^{p_0}\leq P(g;z)$ ($z\in\DD$) for some $g(\theta)\in
L_{\lambda}^{p/p_0}(\partial\DD)$ satisfying $\|g\|_{L_{\lambda}^{p/p_0}(\partial\DD)}\asymp\|f\|^{p_0}_{H^{p}_{\lambda}}$, and hence,
\begin{align}\label{majorization-1}
|f(z)|^{\ell}\leq P(g;z)^{\ell/p_0}, \qquad z\in\DD.
\end{align}
Since for $0\le\tau\le\delta=p^{-1}-\ell^{-1}$, $0\le\tau p_0\le\delta_0$ with $\delta_0=p_0/p-p_0/\ell$, it follows from Theorem \ref{p-means-estimate-b}(i) that,
for $r\in[0,1)$,
\begin{align*}
M_{\tau,\ell}(f;r)^{p_0}\le M_{\tau p_0,\ell/p_0}(g;r)
\lesssim (1-r)^{2\lambda\tau p_0-(2\lambda+1)\delta_0}\|g\|_{L_{\lambda}^{p/p_0}(\partial\DD)},
\end{align*}
so that $M_{\tau,\ell}(f;r)\lesssim (1-r)^{2\lambda\tau-(2\lambda+1)\delta}\|f\|_{H_{\lambda}^{p}}$ for $r\in[0,1)$. Similarly, for $\ell\ge p=p_0$ and $f\in H_{\lambda}^{p}(\DD)$,
appealing to the measure $d\nu\in {\frak B}_{\lambda}(\partial\DD)$ given in Lemma \ref{majorization-a-1}(ii) implies $M_{\tau,\ell}(f;r)^{p_0}\le M_{\tau p_0,\ell/p_0}(d\nu;r)$, and then,
applying Theorem \ref{p-means-estimate-b}(iv) yields
\begin{align*}
M_{\tau,\ell}(f;r)
\lesssim (1-r)^{2\lambda\tau-(2\lambda+1)(p_0^{-1}-\ell^{-1})}\|f\|_{H_{\lambda}^{p_0}},\qquad r\in[0,1).
\end{align*}
Thus part (i) of the theorem is proved.

If $p_0<p<\ell\le+\infty$, by means of (\ref{majorization-1}) and Theorem \ref{p-means-estimate-b}(ii) one has
\begin{align*}
M_{\tau,\ell}(f;r)\le M_{\tau p_0,\ell/p_0}(g;r)^{1/p_0}
\lesssim o\left((1-r)^{2\lambda\tau-(2\lambda+1)\delta}\right),\qquad r\rightarrow1-,
\end{align*}
that concludes part (ii).

As for part (iii), it follows from (\ref{majorization-1}) that the left hand
side of (\ref{p-means-estimate-4}) is dominated by
\begin{eqnarray*}
\left(\int_0^1(1-r)^{(k/p_0)\delta_0(1+2\lambda)-2\lambda(\tau p_0) k/p_0-1}M_{\tau p_0,\ell/p_0}(g;r)^{k/p_0}dr\right)^{1/k}.
\end{eqnarray*}
Since $1<p/p_0<\ell/p_0\le+\infty$, $p/p_0\le k/p_0<+\infty$ and $0\le\tau p_0<\delta_0$, by Theorem \ref{p-means-estimate-b}(iii) the above expression is further dominated by a multiple of $\|g\|_{L_{\lambda}^{p/p_0}(\partial\DD)}^{1/p_0}\asymp\|f\|_{H^{p}_{\lambda}}$. This prove part (iii), and the proof of the theorem is completed.
\end{proof}

For $p_0\le p<1$, taking $\ell=1$ in Theorem \ref{p-means-estimate-c}(i) we have (Lemma \ref{p-means-estimate-a})

\begin{corollary}\label{p-means-estimate-d}
Let $p_{0}\le p<1$, $\delta=p^{-1}-1$ and $0\le\tau\le\delta$. Then for $f\in H_{\lambda}^{p}(\DD)$,
\begin{eqnarray*}
M_{\tau,1}(f;r)\lesssim (1-r)^{2\lambda\tau-(2\lambda+1)\delta}\|f\|_{H_{\lambda}^{p}}, \qquad r\in[0,1).
\end{eqnarray*}
\end{corollary}

\section{Estimates of the derivatives of the $\lambda$-Cauchy kernel $C(z,w)$}

In this section we give a sharp estimate for the derivatives of the $\lambda$-Cauchy kernel $C(z,w)$ (see Theorem \ref{Cauchy-derivative-a} below), which is necessary in evaluating the $(\lambda,\alpha)$-Bergman kernel $K_{\lambda,\alpha}(z,\zeta)$ in the next section.

From \cite{LL1} we have

\begin{lemma} \label{Cauchy-estimate-a}  {\rm(\cite[Theorem 4.2]{LL1})}
For $|zw|<1$, we have
\begin{eqnarray} \label{Cauchy-inequality-1}
|C(z,w)|\lesssim \frac{(|1-zw|+|1-z\bar{w}|)^{-2\lambda}}{|1-z\bar{w}|}\ln\left(\frac{|1-z\bar{w}|^2}{|1-zw|^2}+2\right).
\end{eqnarray}
\end{lemma}

To evaluate the derivatives of $C(z,w)$, we need three lemmas.

\begin{lemma}\label{elementary-derivative-a}
For $k=0,1,\cdots$ and for $z=re^{i\theta}$,
\begin{align*}
\left|\left(r\frac{\partial}{\partial r}\right)^k\left[\frac{1}{|1-zw|^{2\lambda}}\right]\right|
\lesssim \frac{1}{|1-zw|^{2\lambda+k}}.
\end{align*}
\end{lemma}

\begin{proof}
It suffices to prove the following equality
\begin{align} \label{elementary-derivative-2}
\left(r\frac{\partial}{\partial r}\right)^k\left[\frac{1}{|1-zw|^{2\lambda}}\right]
=\sum_{j=0}^{[k/2]}\frac{a_{k,j}(rs)}{|1-zw|^{2(\lambda+j)}}
+\sum_{j=[k/2]+1}^k\frac{(1-r^2s^2)^{2j-k}a_{k,j}(rs)}{|1-zw|^{2(\lambda+j)}}
\end{align}
for $k=1,2,\cdots$, where $z=re^{i\theta}$, $w=se^{i\varphi}$, and $a_{k,j}$ are some polynomials of one variable.

It is easy to find, for $z=re^{i\theta}$,
\begin{align*}
r\frac{d}{dr}\left[\frac{1}{|1-zw|^{2\lambda}}\right]
=\lambda\frac{1-|zw|^2-|1-zw|^2}{|1-zw|^{2\lambda+2}},
\end{align*}
that can be written into the form (\ref{elementary-derivative-2}) with $k=1$.

Assume that (\ref{elementary-derivative-2}) holds for $k$. We are going to prove that it holds too for $k+1$ instead of $k$.
From (\ref{elementary-derivative-2}) we have
\begin{align}\label{elementary-derivative-4}
&\left(r\frac{\partial}{\partial r}\right)^{k+1}\left[\frac{1}{|1-zw|^{2\lambda}}\right]
=\sum_{j=0}^{[k/2]}\frac{\tilde{a}_{k,j}(rs)}{|1-zw|^{2(\lambda+j)}}
+\sum_{j=0}^{[k/2]}(\lambda+j)\frac{(1-r^2s^2)a_{k,j}(rs)}{|1-zw|^{2(\lambda+j+1)}}\\
&\qquad\quad +\sum_{j=[k/2]+1}^k\frac{(1-r^2s^2)^{2j-k-1}\tilde{a}_{k,j}(rs)}{|1-zw|^{2(\lambda+j)}}
+\sum_{j=[k/2]+1}^k(\lambda+j)\frac{(1-r^2s^2)^{2j-k+1}a_{k,j}(rs)}{|1-zw|^{2(\lambda+j+1)}},\nonumber
\end{align}
where
\begin{align*}
\tilde{a}_{k,j}(t)
&=ta'_{k,j}(t)-(\lambda+j)a_{k,j}(t),\qquad 0\le j\le[k/2],\\
\tilde{a}_{k,j}(t)
&=(1-t^2)\left[ta'_{k,j}(t)-(\lambda+j)a_{k,j}(t)\right]-2(2j-k) t^2a_{k,j}(t),\qquad [k/2]+1\le j\le k.
\end{align*}

If $k$ is odd, $[k/2]=(k-1)/2$ and $[(k+1)/2]=[k/2]+1=(k+1)/2$, and the third sum on the right hand side of (\ref{elementary-derivative-4}) can be rewritten into
\begin{align*}
\frac{\tilde{a}_{k,[k/2]}(rs)}{|1-zw|^{2(\lambda+(k+1)/2)}}+\sum_{j=[(k+1)/2]+1}^k\frac{(1-r^2s^2)^{2j-k-1}\tilde{a}_{k,j}(rs)}{|1-zw|^{2(\lambda+j)}}.
\end{align*}
Combining this with the other sums on the right hand side of (\ref{elementary-derivative-4}) proves (\ref{elementary-derivative-2}) for $k+1$ instead of $k$ when $k$ is odd.

If $k$ is even, $[k/2]=[(k-1)/2]+1=k/2$ and $[(k+1)/2]=[k/2]=k/2$, and the last term of the second sum on the right hand side of (\ref{elementary-derivative-4}) can be written into
\begin{align*}
\left(\lambda+\frac{k}{2}\right)\frac{(1-r^2s^2)a_{k,k/2}(rs)}{|1-zw|^{2(\lambda+(k/2)+1)}}.
\end{align*}
The equality (\ref{elementary-derivative-2}) for $k+1$ instead of $k$ is yielded by inserting this into (\ref{elementary-derivative-4}) when $k$ is even.
\end{proof}

Let
\begin{align}\label{A-definition-1}
A(r)=\frac{4({\rm Im} z)({\rm Im} w)} {|1-zw|^{2}},\qquad z=re^{i\theta}.
\end{align}

\begin{lemma}\label{elementary-derivative-b}
For $k=1,2,\cdots$ and for $z=re^{i\theta}$, $w=se^{i\varphi}$,
\begin{align} \label{elementary-derivative-5}
\left(r\frac{\partial}{\partial r}\right)^kA(r)
=A(r)\left[\sum_{j=1}^{[k/2]}\frac{b_{k,j}(rs)}{|1-zw|^{2j}}
+\sum_{j=[k/2]+1}^k\frac{(1-r^2s^2)^{2j-k}a_{k,j}(rs)}{|1-zw|^{2j}}\right],
\end{align}
where $a_{k,j}$, $b_{k,j}$ are some polynomials of one variable.
\end{lemma}

\begin{proof}
The proof is similar to that of (\ref{elementary-derivative-2}).
It is easy to find, for $z=re^{i\theta}$,
\begin{align} \label{elementary-derivative-6}
r\frac{d}{dr}A(r)
=A(r)\frac{1-|zw|^2}{|1-zw|^{2}},
\end{align}
that can be written into the form (\ref{elementary-derivative-5}) with $k=1$. The general case of (\ref{elementary-derivative-5}) is proved by induction.

Assume that (\ref{elementary-derivative-5}) holds for $k$. We are going to prove that it holds too for $k+1$ instead of $k$.
Similarly to (\ref{elementary-derivative-4}), from (\ref{elementary-derivative-5}) and (\ref{elementary-derivative-6}) we have
\begin{align}\label{elementary-derivative-7}
&\frac{1}{A(r)}\left(r\frac{\partial}{\partial r}\right)^{k+1}A(r)
=\sum_{j=1}^{[k/2]}\left[\frac{\tilde{a}_{k,j}(rs)}{|1-zw|^{2j}}
+(j+1)\frac{(1-r^2s^2)a_{k,j}(rs)}{|1-zw|^{2(j+1)}}\right]\\
&\qquad\qquad\qquad +\sum_{j=[k/2]+1}^k \frac{(1-r^2s^2)^{2j-k-1}}{|1-zw|^{2j}} \left[\tilde{a}_{k,j}(rs)
+(j+1)\frac{(1-r^2s^2)^2a_{k,j}(rs)}{|1-zw|^2}\right],\nonumber
\end{align}
where
\begin{align*}
\tilde{a}_{k,j}(t)
&=ta'_{k,j}(t)-ja_{k,j}(t),\qquad 1\le j\le[k/2],\\
\tilde{a}_{k,j}(t)
&=(1-t^2)\left[ta'_{k,j}(t)-ja_{k,j}(t)\right]-2(2j-k) t^2a_{k,j}(t),\qquad [k/2]+1\le j\le k.
\end{align*}
By considering $k$ odd or even respectively as in the proof of Lemma \ref{elementary-derivative-a}, the expression on the right hand side of (\ref{elementary-derivative-7})
can be written in the form within the square brackets on the right hand side of (\ref{elementary-derivative-5}) for $k+1$ instead of $k$. Thus the proof of the lemma is completed.
\end{proof}

\begin{lemma}\label{elementary-derivative-c}
Let $F$ be a smooth function on $(-1,1)$, and $A=A(r):\,[0,1)\mapsto(-1,1)$ a smooth mapping. Then
for $k=1,2,\cdots$,
\begin{align}\label{elementary-derivative-8}
\left(r\frac{\partial}{\partial r}\right)^kF(A)
=& \sum_{j=1}^{[k/2]}\left(\sum_{\substack{i_1+\cdots+i_j=k \\
0\le i_1\le\cdots\le i_j\le k-j+1}}a^{(k,j)}_{i_1,\cdots,i_j}D_r^{i_1}A\cdots D_r^{i_j}A\right) \partial_A^jF(A)\\
& +\sum_{j=0}^{[(k-1)/2]}\left(\sum_{\substack{i_1+\cdots+i_j=2j \\
0\le i_1\le\cdots\le i_j\le j+1}}b^{(k,j)}_{i_1,\cdots,i_j}D_r^{i_1}A\cdots D_r^{i_j}A\right) (D_rA)^{k-2j}\,\partial_A^{k-j}F(A), \nonumber
\end{align}
where
$D_r=r(\partial/\partial_r)$, $D^0_rA=1$, $\partial_A=d/dA$, and $a^{(k,j)}_{i_1,\cdots,i_j}$, $b^{(k,j)}_{i_1,\cdots,i_j}$ are all constants.
\end{lemma}

\begin{proof}
It is obvious that
\begin{align}\label{elementary-derivative-8-1}
r\frac{\partial}{\partial r} F(A)
=D_rA\cdot\frac{d}{dA}F(A),
\end{align}
so that (\ref{elementary-derivative-8}) is true for $k=1$.

Assuming that (\ref{elementary-derivative-8}) holds for $k$, one has
\begin{align}\label{elementary-derivative-9}
&\left(r\frac{\partial}{\partial r}\right)^{k+1}F(A)
=\sum_{j=1}^{[k/2]}\left(\sum_{\substack{i_1+\cdots+i_j=k \\
0\le i_1\le\cdots\le i_j\le k-j+1}}a^{(k,j)}_{i_1,\cdots,i_j}D_r^{i_1}A\cdots D_r^{i_j}A\right) D_rA\,\partial_A^{j+1}F(A) \\
&\qquad +\sum_{j=1}^{[k/2]}\left(\sum_{\substack{i_1+\cdots+i_j=k \\
0\le i_1\le\cdots\le i_j\le k-j+1}}a^{(k,j)}_{i_1,\cdots,i_j} \sum_{\ell=1}^j D_r^{i_1}A\cdots D_r\left(D_r^{i_{\ell}}A\right)\cdots D_r^{i_j}A\right) \partial_A^jF(A) \nonumber\\
&\qquad +\sum_{j=0}^{\left[\frac{k-1}{2}\right]}\left(\sum_{\substack{i_1+\cdots+i_j=2j \\
0\le i_1\le\cdots\le i_j\le j+1}}b^{(k,j)}_{i_1,\cdots,i_j}D_r^{i_1}A\cdots D_r^{i_j}A\right) (D_rA)^{k+1-2j}\,\partial_A^{k+1-j}F(A) \nonumber\\
&\qquad +\sum_{j=0}^{\left[\frac{k-1}{2}\right]}\left(\sum_{\substack{i_1+\cdots+i_j=2j \\
0\le i_1\le\cdots\le i_j\le j+1}}b^{(k,j)}_{i_1,\cdots,i_j}D_r^{i_1}A\cdots D_r^{i_j}A\right) (k-2j)(D_r^2A)(D_rA)^{k-2j-1}\partial_A^{k-j}F(A) \nonumber\\
&\qquad +\sum_{j=0}^{\left[\frac{k-1}{2}\right]}\left(\sum_{\substack{i_1+\cdots+i_j=2j \\
0\le i_1\le\cdots\le i_j\le j+1}}b^{(k,j)}_{i_1,\cdots,i_j} \sum_{\ell=1}^j D_r^{i_1}A\cdots D_r\left(D_r^{i_{\ell}}A\right)\cdots D_r^{i_j}A\right) (D_rA)^{k-2j}\,\partial_A^{k-j}F(A).\nonumber
\end{align}
If $k$ is odd, $[k/2]=(k-1)/2$ and $[(k+1)/2]=[k/2]+1=(k+1)/2$, and the last term in the fourth sum on the right hand side of (\ref{elementary-derivative-9}) is
\begin{align*}
\left(\sum_{\substack{i_1+\cdots+i_{\frac{k+1}{2}}=k+1 \\
0\le i_1\le\cdots\le i_{\frac{k+1}{2}}\le \frac{k+1}{2}+1}}\tilde{b}^{(k,\frac{k+1}{2})}_{i_1,\cdots,i_{\frac{k+1}{2}}}D_r^{i_1}A\cdots D_r^{i_{\frac{k+1}{2}}}A\right) \partial_A^{\frac{k+1}{2}}F(A),
\end{align*}
which is identical with the last term of the first sum on the right hand side of (\ref{elementary-derivative-8}), with $k+1$ instead of $k$ and $j=(k+1)/2=[k/2]+1$,
and the remaining terms in the fourth sum on the right hand side of (\ref{elementary-derivative-9}) can be rewritten into
\begin{align*}
\sum_{j=0}^{\left[\frac{k}{2}\right]-1}\left(\sum_{\substack{i_1+\cdots+i_{j+1}=2(j+1) \\
0\le i_1\le\cdots\le i_{j+1}\le j+2}}\tilde{b}^{(k,j)}_{i_1,\cdots,i_{j+1}}D_r^{i_1}A\cdots D_r^{i_{j+1}}A\right) (D_rA)^{k+1-2(j+1)}\partial_A^{k+1-(j+1)}F(A)\\
\end{align*}
which is identical with, by letting $j'=j+1$, the second sum on the right hand side of (\ref{elementary-derivative-8}), with $k+1$ instead of $k$.
For the last sum on the right hand side of (\ref{elementary-derivative-9}), we write $(D_rA)^{k-2j}\,\partial_A^{k-j}F(A)=D_rA\,(D_rA)^{k+1-2(j+1)}\,\partial_A^{k+1-(j+1)}F(A)$,
and then, treat its last term and remaining terms by the same way as above to arrive at similar assertions. It is easy to find that, the first two sums on the right hand side of (\ref{elementary-derivative-9}) can be converted into the form of the first sum on the right hand side of (\ref{elementary-derivative-8}) with $k+1$ instead of $k$, and the third sum is part of the second sum on the right hand side of (\ref{elementary-derivative-8}), again with $k+1$ instead of $k$.
This proves (\ref{elementary-derivative-8}) for $k+1$ instead of $k$ when $k$ is odd.

If $k$ is even, $[k/2]=[(k-1)/2]+1=k/2$ and $[(k+1)/2]=[k/2]=k/2$, and the last three sums on the right hand side of (\ref{elementary-derivative-9}) can be converted into the form of the second sum on the right hand side of (\ref{elementary-derivative-8}) with $k+1$ instead of $k$. Furthermore, the last term in the first sum on the right hand side of (\ref{elementary-derivative-9}) is
\begin{align*}
\left(\sum_{\substack{i_1+\cdots+i_{\frac{k}{2}}=k \\
0\le i_1\le\cdots\le i_{\frac{k}{2}}\le \frac{k}{2}+1}}a^{(k,\frac{k}{2})}_{i_1,\cdots,i_{\frac{k}{2}}}D_r^{i_1}A\cdots D_r^{i_{\frac{k}{2}}}A\right) D_rA\,\partial_A^{\frac{k}{2}+1}F(A),
\end{align*}
which is identical with the last term of the second sum on the right hand side of (\ref{elementary-derivative-8}), with $k+1$ instead of $k$ and $j=[k/2]=k/2$. We again prove (\ref{elementary-derivative-8}) for $k+1$ instead of $k$ when $k$ is even.
\end{proof}

%
%

\begin{theorem}\label{Cauchy-derivative-a}
For $m=1,2,\cdots$ and for $z=re^{i\theta},w\in\DD$,
\begin{align}\label{Cauchy-derivative-1}
\left|\left(r\frac{\partial}{\partial r}\right)^m\left[C(z,w)\right]\right|
\lesssim \frac{(|1-z\overline{w}|+|1-zw|)^{-2\lambda}}{|1-z\overline{w}|}
\left(\frac{1}{|1-z\overline{w}|^m}+\frac{1}{|1-zw|^m}\right).
\end{align}
\end{theorem}

\begin{proof}
{\bf Case I:} $|1-z\bar{w}|^2<2|1-zw|^2$.

For this case, it follows from (\ref{Cauchy-kernel-2-2}) and the first equality in (\ref{Poi-0}) that
\begin{align*}
C(z,w)&=\frac{1}{1-z\bar{w}}\frac{1}{|1-zw|^{2\lambda}}F(A(r)),\qquad |zw|<1,
\end{align*}
where
\begin{align*}
F(A)={}_2\!F_{1}(\lambda,\lambda;\, 2\lambda+1;\, A),
\end{align*}
and $A=A(r)$ is given by (\ref{A-definition-1}).
Since
\begin{align}\label{A-equality-1}
1-A(r)=\frac{|1-z\bar{w}|^2}{|1-zw|^2},
\end{align}
it follows that $-1<A(r)<1$ for $|1-z\bar{w}|^2<2|1-zw|^2$, and $F(A(r))$ is bounded.

With the notation $D_r=r(\partial/\partial_r)$, Leibniz's rule gives
\begin{align*}
\left|D_r^m\left[C(z,w)\right]\right|
\lesssim \sum_{j+\ell+k=m} \left|D_r^j\left[\frac{1}{1-z\bar{w}}\right]\right|
\left|D_r^{\ell}\left[\frac{1}{|1-zw|^{2\lambda}}\right]\right|
\left|D_r^k\left[F(A(r))\right]\right|,\qquad |zw|<1.
\end{align*}
By induction, one has
$$
D_r^j\left[\frac{1}{1-z\bar{w}}\right]=\sum_{\tau=1}^{j+1}\frac{a_{j,\tau}}{(1-z\bar{w})^{\tau}}.
$$
Appealing to this and Lemma \ref{elementary-derivative-a}, we get
\begin{align}\label{Cauchy-derivative-2}
\left|D_r^m\left[C(z,w)\right]\right|
\lesssim \sum_{j+\ell+k=m} \frac{1}{|1-z\bar{w}|^{j+1}}
\frac{1}{|1-zw|^{2\lambda+\ell}}
\left|D_r^k\left[F(A(r))\right]\right|,\qquad |zw|<1.
\end{align}

Since (cf. \cite[2-1(7)]{Er})
\begin{align}\label{Gauss-2}
\partial_AF(A)
=\frac{\lambda^2}{2\lambda+1}\, {}_2\!F_{1}\left(\lambda+1,\lambda+1;2\lambda+2;A\right)
\end{align}
and
\begin{eqnarray} \label{Gauss-1-1}
{}_2\!F_{1}(a,b;a+b;t)\asymp\ln\left(\frac{1}{1-t}+2\right),\qquad -1\le t<1,
\end{eqnarray}
from (\ref{elementary-derivative-6}), (\ref{elementary-derivative-8-1}) and (\ref{A-equality-1}) we have
\begin{align}\label{elementary-derivative-10}
\left|D_r\left[F(A(r))\right]\right|
\lesssim \frac{1}{|1-zw|}\ln\left(\frac{|1-zw|^2}{|1-z\bar{w}|^2}+2\right)
\lesssim \frac{1}{|1-z\overline{w}|},\qquad z=re^{i\theta},
\end{align}
where the last inequality is based on the fact $s\ln\left(s^{-2}+2\right)\lesssim1$ for $s\in(0,2]$.
Thus from (\ref{Cauchy-derivative-2}) we get
\begin{align*}
\left|D_r\left[C(z,w)\right]\right|
&\lesssim \frac{|1-zw|^{-2\lambda}}{|1-z\bar{w}|^2} +\frac{|1-zw|^{-2\lambda-1}}{|1-z\bar{w}|}
+\frac{|1-zw|^{-2\lambda}}{|1-z\bar{w}|}\left|D_r\left[F(A(r))\right]\right|\\
&\lesssim \frac{|1-zw|^{-2\lambda}}{|1-z\bar{w}|^2},\qquad |zw|<1,
\end{align*}
so that (\ref{Cauchy-derivative-1}) is true for $m=1$ and for $|1-z\bar{w}|^2<2|1-zw|^2$.

In order to show (\ref{Cauchy-derivative-1}) for $m\ge2$, at first by Lemma \ref{elementary-derivative-b}, $\left|D_r^kA(r)\right|\lesssim|1-zw|^{-k}$, and then, by Lemma \ref{elementary-derivative-c} and (\ref{elementary-derivative-6}),
\begin{align}\label{elementary-derivative-11}
\left|D_r^k\left[F(A(r))\right]\right|
\lesssim & \frac{1}{|1-zw|^{k}}\sum_{j=1}^{[k/2]}\left|\partial_A^jF(A)\right| \nonumber\\
& \qquad  +\sum_{j=0}^{[(k-1)/2]}\frac{1}{|1-zw|^{2j}} \left(\frac{1-|zw|^2}{|1-zw|^{2}}\right)^{k-2j}\left|\partial_A^{k-j}F(A)\right|.
\end{align}

For $j\ge2$, again by \cite[2-1(7)]{Er}) we have
$$
\partial_A^jF(A)=\frac{(\lambda)_j(\lambda)_j}{(2\lambda+1)_j} {}_2\!F_{1}\left(\lambda+j,\lambda+j;2\lambda+j+1;A\right),
$$
and further, by \cite[2-1(23)]{Er}),
\begin{align*}
\partial_A^jF(A) &=\frac{(\lambda)_j(\lambda)_j}{(2\lambda+1)_j} (1-A)^{1-j} {}_2\!F_{1}\left(\lambda+1,\lambda+1;2\lambda+j+1;A\right)\\
&\asymp(1-A)^{1-j}.
\end{align*}
Thus for $k\ge2$, from (\ref{A-equality-1}), (\ref{Gauss-2}), (\ref{Gauss-1-1}) and (\ref{elementary-derivative-11}) we have
\begin{align}\label{elementary-derivative-12}
\left|D_r^k\left[F(A(r))\right]\right|
\lesssim & \frac{1}{|1-zw|^{k}}\left[\left(\frac{|1-z\bar{w}|^2}{|1-zw|^2}\right)^{1-[k/2]}+\ln\left(\frac{|1-zw|^2}{|1-z\bar{w}|^2}+2\right)\right] \nonumber\\
& \qquad  +\sum_{j=0}^{[(k-1)/2]}\frac{1}{|1-zw|^{2j}} \left(\frac{1-|zw|^2}{|1-zw|^{2}}\right)^{k-2j} \left(\frac{|1-z\bar{w}|^2}{|1-zw|^2}\right)^{j+1-k} \nonumber\\
\lesssim & \frac{1}{|1-zw|^{k}}\left[\left(\frac{|1-z\bar{w}|}{|1-zw|}\right)^{2-k}+\ln\left(\frac{|1-zw|^2}{|1-z\bar{w}|^2}+2\right)\right] \nonumber\\
& \qquad  +\sum_{j=0}^{[(k-1)/2]}\frac{|1-z\bar{w}|^2}{|1-zw|^2} \left(\frac{1-|zw|^2}{|1-z\bar{w}|}\right)^{k-2j} \frac{1}{|1-z\bar{w}|^k} \nonumber\\
\lesssim & \frac{1}{|1-z\bar{w}|^k}.
\end{align}
Applying (\ref{elementary-derivative-10}) and (\ref{elementary-derivative-12}) to (\ref{Cauchy-derivative-2}) yields
\begin{align*}
\left|D_r^m\left[C(z,w)\right]\right|
&\lesssim \sum_{j+\ell+k=m} \frac{1}{|1-z\bar{w}|^{j+1}}
\frac{1}{|1-zw|^{2\lambda+\ell}}\frac{1}{|1-z\bar{w}|^k}\\
&\lesssim \frac{|1-zw|^{-2\lambda}}{|1-z\bar{w}|^{m+1}},\qquad |zw|<1.
\end{align*}
Thus (\ref{Cauchy-derivative-1}) is proved for $|1-z\bar{w}|^2<2|1-zw|^2$.

{\bf Case II:} $|1-z\bar{w}|^2\ge2|1-zw|^2$.

For this case, it follows from (\ref{Cauchy-kernel-2-2}) and the second equality in (\ref{Poi-0}) that
\begin{align*}
C(z,w)&=\frac{1}{1-z\bar{w}}\frac{1}{|1-z\bar{w}|^{2\lambda}}\tilde{F}(\tilde{A}(r)),\qquad |zw|<1,
\end{align*}
where
\begin{align*}
\tilde{F}(\tilde{A})={}_2\!F_{1}\Big({\lambda,\lambda+1
\atop
  2\lambda+1};\tilde{A}\Big),
\qquad \tilde{A}:=\tilde{A}(r)=-\frac{4({\rm Im} z)({\rm Im}
  w)}{|1-z\bar{w}|^{2}},\qquad  z=re^{i\theta}.
\end{align*}

Since
\begin{align}\label{A-equality-2}
1-\tilde{A}(r)=\frac{|1-zw|^2}{|1-z\bar{w}|^2},
\end{align}
it follows that $1/2\le\tilde{A}(r)<1$ for $|1-z\bar{w}|^2\ge2|1-zw|^2$, and from (\ref{Gauss-1-1}),
\begin{align}\label{Gauss-function-3}
\tilde{F}(\tilde{A}(r))\asymp\ln\left(\frac{|1-z\bar{w}|^2}{|1-zw|^2}+2\right),\qquad  z=re^{i\theta},\,w\in\DD.
\end{align}

Similarly to (\ref{Cauchy-derivative-2}), one has
\begin{align}\label{Cauchy-derivative-3}
\left|D_r^m\left[C(z,w)\right]\right|
\lesssim \sum_{j+\ell+k=m}
\frac{1}{|1-z\bar{w}|^{2\lambda+j+\ell+1}}
\left|D_r^k\left[\tilde{F}(\tilde{A}(r))\right]\right|,\qquad |zw|<1.
\end{align}
Again by Lemma \ref{elementary-derivative-b}, $\left|D_r^k\tilde{A}(r)\right|\lesssim|1-z\bar{w}|^{-k}$, and then, by Lemma \ref{elementary-derivative-c} and (\ref{elementary-derivative-6}),
\begin{align}\label{elementary-derivative-13}
\left|D_r^k\left[\tilde{F}(\tilde{A}(r))\right]\right|
\lesssim & \frac{1}{|1-z\bar{w}|^{k}}\sum_{j=1}^{[k/2]}\left|\partial_{\tilde{A}}^j\tilde{F}(\tilde{A})\right| \nonumber\\
& \qquad  +\sum_{j=0}^{[(k-1)/2]}\frac{1}{|1-z\bar{w}|^{2j}} \left(\frac{1-|zw|^2}{|1-z\bar{w}|^{2}}\right)^{k-2j} \left|\partial_{\tilde{A}}^{k-j}\tilde{F}(\tilde{A})\right|.
\end{align}

For $j\ge1$, also by \cite[2-1(7) and 2-1(23)]{Er}) we have
\begin{align*}
\partial_{\tilde{A}}^j\tilde{F}(\tilde{A}) & =\frac{(\lambda)_j(\lambda+1)_j}{(2\lambda+1)_j} {}_2\!F_{1}\left(\lambda+j,\lambda+j+1;2\lambda+j+1;\tilde{A}\right)\\
&=\frac{(\lambda)_j(\lambda+1)_j}{(2\lambda+1)_j} (1-\tilde{A})^{-j} {}_2\!F_{1}\left(\lambda+1,\lambda;2\lambda+j+1;\tilde{A}\right)\\
&\asymp(1-\tilde{A})^{-j}.
\end{align*}
Thus for $k\ge1$, from (\ref{A-equality-2}) and (\ref{elementary-derivative-13}) we have
\begin{align*}
\left|D_r^k\left[\tilde{F}(\tilde{A}(r))\right]\right|
\lesssim & \frac{1}{|1-z\bar{w}|^{k}}\left(\frac{|1-zw|^2}{|1-z\bar{w}|^2}\right)^{-[k/2]} \nonumber\\
& \qquad  +\sum_{j=0}^{[(k-1)/2]}\frac{1}{|1-z\bar{w}|^{2j}} \left(\frac{1-|zw|^2}{|1-z\bar{w}|^{2}}\right)^{k-2j} \left(\frac{|1-zw|^2}{|1-z\bar{w}|^2}\right)^{j-k} \nonumber\\
\lesssim & \frac{1}{|1-z\bar{w}|^{k}}\left(\frac{|1-zw|}{|1-z\bar{w}|}\right)^{-k}
 +\sum_{j=0}^{[(k-1)/2]}\left(\frac{1-|zw|^2}{|1-zw|}\right)^{k-2j} \frac{1}{|1-zw|^k} \nonumber\\
\lesssim & \frac{1}{|1-zw|^k}.
\end{align*}
Applying this and (\ref{Gauss-function-3}) to (\ref{Cauchy-derivative-3}) yields
\begin{align*}
\left|D_r^m\left[C(z,w)\right]\right|
&\lesssim \frac{1}{|1-z\bar{w}|^{2\lambda+m+1}}\ln\left(\frac{|1-z\bar{w}|^2}{|1-zw|^2}+2\right)
 +\sum_{j+\ell+k=m}
\frac{|1-zw|^{-k}}{|1-z\bar{w}|^{2\lambda+j+\ell+1}} \\
&\lesssim \frac{|1-z\bar{w}|^{-2\lambda}}{|1-z\bar{w}||1-zw|^{m}},\qquad |zw|<1.
\end{align*}
This proves (\ref{Cauchy-derivative-1}) for $|1-z\bar{w}|^2\ge2|1-zw|^2$. The proof of Theorem \ref{Cauchy-derivative-a} is completed.
\end{proof}

\section{Estimates of the $(\lambda,\alpha)$-Bergman kernel $K_{\lambda,\alpha}(z,\zeta)$}


In comparison to the $(\lambda,\alpha)$-Bergman kernel $K_{\lambda,\alpha}(z,w)$, we often use the following simpler form
\begin{align}\label{h-kernel-2-1}
h_{\lambda,\beta}(z,w)=\sum (n+1)^{\beta}\phi_{n}(z)\overline{\phi_{n}(w)}.
\end{align}
But in what follows we work with the following variant of $K_{\lambda,\alpha}(z,w)$
\begin{align}\label{K-kernel-2-1}
\tilde{K}_{\lambda,\beta}(z,w)=\sum_{n=0}^{\infty}
\frac{\Gamma(n+\beta+2\lambda+2)}{\Gamma(n+2\lambda+2)}
\phi_{n}(z)\overline{\phi_{n}(w)}.
\end{align}
By means of (\ref{phi-bound-1}), the series in defining $K_{\lambda,\alpha}(z,\zeta)$, $h_{\lambda,\beta}(z,w)$ and $\tilde{K}_{\lambda,\beta}(z,w)$ are all absolutely and uniformly
convergent for $|zw|\le r_0$ with fixed $0<r_0<1$. Note that $\tilde{K}_{\lambda,0}(z,w)=C(z,w)$.

\begin{lemma}\label{K-kernal-2-a}
{\rm (i)} If $m$ is a nonnegative integer, then there exist constants $a_{m,j}$, $0\le j\le m$, such that
\begin{align*}
\tilde{K}_{\lambda,m}(z,w)=\sum_{j=0}^m a_{m,j}\left(r\frac{\partial}{\partial r}\right)^j\left[C(z,w)\right],\qquad z=re^{i\theta};
\end{align*}

{\rm (ii)} If $-2\lambda-2<\beta<0$, or $\beta>0$ is not an integer, and $m=\max\{0,[\beta]+1\}$, then
\begin{align}\label{K-kernel-2-4}
\tilde{K}_{\lambda,\beta}(z,w)=\frac{1}{\Gamma(m-\beta)}\int_0^1(1-t)^{m-\beta-1}t^{\beta+2\lambda+1}\tilde{K}_{\lambda,m}(tz,w)\,dt.
\end{align}
\end{lemma}

\begin{proof}
For $\beta>0$, one has the following iteration
\begin{align}\label{K-kernel-2-2}
\tilde{K}_{\lambda,\beta}(z,w)=r\frac{\partial}{\partial r}\left[\tilde{K}_{\lambda,\beta-1}(z,w)\right]
+(\beta+2\lambda+1)\tilde{K}_{\lambda,\beta-1}(z,w),\qquad z=re^{i\theta}.
\end{align}
Part (i) follows from (\ref{K-kernel-2-2}) and by induction, and part (ii) is verified by termwise integration since $m>\beta>-2\lambda-2$.
\end{proof}

\begin{theorem}\label{K-kernal-b}
{\rm (i)} For $\beta>0$,
\begin{align*}
|\tilde{K}_{\lambda,\beta}(z,w)|\lesssim \frac{(|1-z\overline{w}|+|1-zw|)^{-2\lambda}}{|1-z\overline{w}|}
\left(\frac{1}{|1-z\overline{w}|^{\beta}}+\frac{1}{|1-zw|^{\beta}}\right),\quad z,w\in\DD;
\end{align*}

{\rm (ii)} for $\beta=0$,
\begin{align*}
|\tilde{K}_{\lambda,0}(z,w)|\lesssim \frac{(|1-z\overline{w}|+|1-zw|)^{-2\lambda}}{|1-z\overline{w}|}
\ln\left(\frac{|1-z\overline{w}|}{|1-zw|}+2\right),\quad z,w\in\DD;
\end{align*}

{\rm (iii)} for $-1<\beta<0$,
\begin{align*}
|\tilde{K}_{\lambda,\beta}(z,w)|\lesssim \frac{(|1-z\overline{w}|+|1-zw|)^{-2\lambda}}{|1-z\overline{w}|^{\beta+1}},\quad z,w\in\DD;
\end{align*}

{\rm (iv)} for $\beta=-1$,
\begin{align*}
|\tilde{K}_{\lambda,-1}(z,w)|\lesssim (|1-z\overline{w}|+|1-zw|)^{-2\lambda}\ln\left(\frac{|1-zw|}{|1-z\overline{w}|}+2\right),\quad z,w\in\DD;
\end{align*}

{\rm (v)} for $-2\lambda-1<\beta<-1$,
\begin{align*}
|\tilde{K}_{\lambda,\beta}(z,w)|\lesssim (|1-z\overline{w}|+|1-zw|)^{-\beta-2\lambda-1},\quad z,w\in\DD;
\end{align*}

{\rm (vi)} for $\beta=-2\lambda-1$,
\begin{align*}
|\tilde{K}_{\lambda,\beta}(z,w)|\lesssim \ln\left(\frac{1}{|1-z\overline{w}|+|1-zw|}+2\right),\quad z,w\in\DD;
\end{align*}

{\rm (vii)} for $-2\lambda-2<\beta<-2\lambda-1$, $\tilde{K}_{\lambda,\beta}(z,w)$ is continuous for $z,w$ satisfying $|zw|\le1$.
\end{theorem}

\begin{proof}
By Lemmas \ref{Cauchy-estimate-a}, \ref{K-kernal-2-a}(i), and Theorem \ref{Cauchy-derivative-a}, for a positive integer $m$ we have
\begin{align}\label{K-kernel-2-11}
\left|\tilde{K}_{\lambda,m}(z,w)\right|
 \lesssim \frac{(|1-zw|+|1-z\bar{w}|)^{-2\lambda}}{|1-z\bar{w}|}
 \left[\frac{1}{|1-z\overline{w}|^m}+\frac{1}{|1-zw|^m}\right], \qquad z,w\in\DD.
\end{align}
Note the logarithmic function appearing in (\ref{Cauchy-inequality-1}) is certainly controlled by the expression within the square brackets above.

If $\beta>0$ is not an integer and $m=[\beta]+1$, it follows from (\ref{K-kernel-2-4}) and (\ref{K-kernel-2-11}) that
\begin{align*}
\left|\tilde{K}_{\lambda,\beta}(z,w)\right|
\lesssim \int_0^1 \frac{(1-t)^{m-\beta-1}t^{\beta+2\lambda+1}}{(|1-tzw|+|1-tz\bar{w}|)^{2\lambda}|1-tz\bar{w}|}
 \left[\frac{1}{|1-tz\overline{w}|^m}+\frac{1}{|1-tzw|^m}\right]\,dt.
\end{align*}
Since, similarly to (\ref{elementary-equality-1}),
\begin{align}\label{elementary-equality-1-1}
|1-tz|\asymp 1-t+1-r+\left|\sin\theta/2\right| \asymp 1-t+|1-z|,\qquad z=re^{i\theta}\in\DD,\,\,t\in[0,1),
\end{align}
we have
\begin{align*}
\left|\tilde{K}_{\lambda,\beta}(z,w)\right|
\lesssim &\frac{|1-z\bar{w}|^{-1}}{(|1-zw|+|1-z\bar{w}|)^{2\lambda}} \nonumber\\
& \qquad\qquad \times\left[\int_0^1\frac{(1-t)^{m-\beta-1}dt}{(1-t+|1-z\overline{w}|)^m}
+\int_0^1\frac{(1-t)^{m-\beta-1}dt}{(1-t+|1-zw|)^m}\right]\nonumber\\
\lesssim &\frac{|1-z\bar{w}|^{-1}}{(|1-zw|+|1-z\bar{w}|)^{2\lambda}}\left[\frac{1}{|1-z\overline{w}|^{\beta}}
+\frac{1}{|1-zw|^{\beta}}\right].
\end{align*}
Part (i) of the theorem is proved.

Since $\tilde{K}_{\lambda,0}(z,w)=C(z,w)$, part (ii) is consistent with Lemma \ref{Cauchy-estimate-a}.

If $-2\lambda-1<\beta<0$, it follows from (\ref{K-kernel-2-4}) and part (ii) that
\begin{align*}
\left|\tilde{K}_{\lambda,\beta}(z,w)\right|
\lesssim \int_0^1 \frac{(1-t)^{-\beta-1}t^{\beta+2\lambda+1}}{(|1-tzw|+|1-tz\bar{w}|)^{2\lambda}|1-tz\bar{w}|}
 \ln\left(\frac{|1-tz\overline{w}|}{|1-tzw|}+2\right)\,dt,
\end{align*}
and on account of (\ref{elementary-equality-1-1}),
\begin{align}\label{K-kernel-2-12}
\left|\tilde{K}_{\lambda,\beta}(z,w)\right|
\lesssim \int_0^1 \frac{\ln\left(\frac{1-t+|1-z\overline{w}|}{1-t+|1-zw|}+2\right)(1-t)^{-\beta-1}}
{(1-t+|1-zw|+|1-z\bar{w}|)^{2\lambda}(1-t+|1-z\bar{w}|)}
\,dt,
\end{align}

If $|1-z\bar{w}|\le|1-zw|$, from (\ref{K-kernel-2-12}) one has
\begin{align}\label{K-kernel-2-13}
\left|\tilde{K}_{\lambda,\beta}(z,w)\right|
& \lesssim \int_0^1 \frac{(1-t)^{-\beta-1}}
{(1-t+|1-zw|)^{2\lambda}(1-t+|1-z\bar{w}|)}\,dt \nonumber\\
& =\frac{1}{|1-z\bar{w}|^{2\lambda+\beta+1}} \int_0^{B_1} \frac{s^{-\beta-1}}
{(s+B_2)^{2\lambda}(s+1)}\,ds,
\end{align}
where
$$
B_1=\frac{1}{|1-z\bar{w}|},\qquad B_2=\frac{|1-zw|}{|1-z\bar{w}|}.
$$
We divide the last integral in (\ref{K-kernel-2-13}) into two parts over the intervals $[0,B_2/2]$ and $[B_2/2,B_1]$ respectively,
where the second one is obviously dominated by $\int_{B_2/2}^{B_1}s^{-2\lambda-\beta-2}ds\lesssim B_2^{-2\lambda-\beta-1}$, and the first one
by $B_2^{-2\lambda}\int_0^{B_2/2} s^{-\beta-1}ds/(s+1)$, and further by a multiple of $B_2^{-2\lambda}$ if $-1<\beta<0$; $B_2^{-2\lambda}\ln(B+2)$ if $\beta=-1$; and $B_2^{-2\lambda-\beta-1}$  if $-2\lambda-1<\beta<-1$.
Inserting these estimates into (\ref{K-kernel-2-13}) proves parts (iii), (iv) and (v) for $|1-z\bar{w}|\le|1-zw|$ simultaneously.

If $|1-z\bar{w}|>|1-zw|$, from (\ref{K-kernel-2-12}) one has
\begin{align}\label{K-kernel-2-14}
\left|\tilde{K}_{\lambda,\beta}(z,w)\right|
& \lesssim \int_0^1 \frac{(1-t)^{-\beta-1}}
{(1-t+|1-z\bar{w}|)^{2\lambda+1}} \ln\left(\frac{|1-z\overline{w}|}{1-t}+2\right)\,dt \nonumber\\
& =\frac{1}{|1-z\bar{w}|^{2\lambda+\beta+1}} \int_0^{B_1} \frac{s^{-\beta-1}}{(s+1)^{2\lambda+1}} \ln\left(\frac{1}{s}+2\right)\,ds,
\end{align}
which proves parts (iii), (iv) and (v) again since the last integral above is bounded.

For part (vi), with $\beta=-2\lambda-1$, (\ref{K-kernel-2-12}) holds still; and when $|1-z\bar{w}|\le|1-zw|$, (\ref{K-kernel-2-13}) becomes
\begin{align*}
\left|\tilde{K}_{\lambda,\beta}(z,w)\right|
\lesssim\int_0^{B_1} \frac{s^{2\lambda}}{(s+B_2)^{2\lambda}(s+1)}\,ds.
\end{align*}
Evaluating this integral by the same way as before, we have
\begin{align*}
\left|\tilde{K}_{\lambda,\beta}(z,w)\right|
\lesssim B_2^{-2\lambda}\int_0^{B_2/2} s^{2\lambda-1}\,ds
+\int_{B_2/2}^{B_1} \frac{1}{s}\,ds
\lesssim \ln\frac{2}{|1-zw|}
\end{align*}

When $|1-z\bar{w}|>|1-zw|$, (\ref{K-kernel-2-14}) becomes
\begin{align*}
\left|\tilde{K}_{\lambda,\beta}(z,w)\right|
=\int_0^{B_1} \frac{s^{2\lambda}}{(s+1)^{2\lambda+1}} \ln\left(\frac{1}{s}+2\right)\,ds,
\lesssim \ln\left(\frac{1}{|1-z\bar{w}|}+1\right).
\end{align*}
Combining the above two cases proves part (vi).

Part (vii) is a direct consequence of the definition (\ref{K-kernel-2-1}) of $\tilde{K}_{\lambda,\beta}(z,w)$ and (\ref{phi-bound-1}). The proof of the theorem is completed.
\end{proof}

\begin{corollary}\label{h-kernel-2-a}
The function $h_{\lambda,\beta}(z,w)$ defined by (\ref{h-kernel-2-1}) has the same estimates as those for $\tilde{K}_{\lambda,\beta}(z,w)$ given in Theorem \ref{K-kernal-b}(i)-(vi).
Moreover, for $\beta<-2\lambda-1$, $h_{\lambda,\beta}(z,w)$ is continuous for $z,w$ satisfying $|zw|\le1$.
\end{corollary}

The last assertion in the corollary for $\beta<-2\lambda-1$ is a direct consequence of (\ref{phi-bound-1}). If $\beta\ge-2\lambda-1$, one has
\begin{align*}
(n+1)^{\beta}=\sum_{j=0}^{M}a_{\beta,j}
\frac{\Gamma(n+\beta-j+2\lambda+2)}{\Gamma(n+2\lambda+2)}+O\left((n+1)^{\beta-M-1}\right)
\end{align*}
for $n\ge0$, where $M=[\beta+2\lambda+1]$, and so the corollary follows from Theorem \ref{K-kernal-b} immediately.

\begin{theorem}\label{Bergman-type-kernel-a}
For $\lambda\ge0$ and $\alpha>-1$, the $(\lambda,\alpha)$-Bergman kernel $K_{\lambda,\alpha}(z,w)$ has the following estimate
\begin{align}\label{Bergman-type-kernel-2}
|K_{\lambda,\alpha}(z,w)|\lesssim \frac{|1-zw|^{-1}}{(|1-z\overline{w}|+|1-zw|)^{2\lambda}}
\left(\frac{1}{|1-z\overline{w}|^{\alpha+1}}+\frac{1}{|1-zw|^{\alpha+1}}\right),\quad z,w\in\DD.
\end{align}
\end{theorem}

The theorem is a consequence of Corollary \ref{h-kernel-2-a}, since
\begin{align}\label{elementary-equality-4}
\frac{\Gamma(\lambda+1)\Gamma(n+\lambda+\alpha+2)}{\Gamma(\lambda+\alpha+2)\Gamma(n+\lambda+1)}=\sum_{j=0}^{M}\tilde{a}_{\alpha,j}(n+1)^{\alpha+1-j}+O\left((n+1)^{\alpha-M}\right)
\end{align}
for $n\ge0$, where $M=[\alpha+2\lambda+2]$.

We now come to the proof of Lemma \ref{kernel-p-mean-a}, which is restated as follows.

\begin{theorem}\label{kernel-p-mean-b}
For $0<p\le1$ and $\alpha>(2\lambda+1)(p^{-1}-1)-1$,
\begin{align}\label{kernel-p-mean-2}
\int_{-\pi}^{\pi}|K_{\lambda,\alpha}(z,se^{i\varphi})|^{p}\,dm_{\lambda}(\varphi)
\lesssim \frac{|1-s^{2}z^{2}|^{2\lambda(1-p)}}{(1-s|z|)^{(2+\alpha)p-1}},\qquad z\in\DD,\,\,\,s\in[0,1).
\end{align}
\end{theorem}

\begin{proof}
Let $z=re^{i\theta}$ with $\theta\in[0,\pi]$. If $\theta\in[-\pi,0]$, the conclusion follows from the equality $K_{\lambda,\alpha}(z,se^{i\varphi})=\overline{K_{\lambda,\alpha}(\bar{z},se^{-i\varphi})}$.
We first have
\begin{align*}
\int_{-\pi}^{\pi}|K_{\lambda,\alpha}(z,se^{i\varphi})|^{p}\,dm_{\lambda}(\varphi)
\lesssim \int_0^{\pi}\left(|K_{\lambda,\alpha}(z,se^{i\varphi})| +|K_{\lambda,\alpha}(z,se^{-i\varphi})|\right)^{p} \,dm_{\lambda}(\varphi),
\end{align*}
and then, by Theorem \ref{Bergman-type-kernel-a}, this is further dominated by
\begin{align*}
\int_0^{\pi}
\left(\frac{1}{|1-z\overline{w}|^{\alpha+2}}+\frac{1}{|1-zw|^{\alpha+2}}\right)^p
\frac{dm_{\lambda}(\varphi)}{(|1-z\overline{w}|+|1-zw|)^{2\lambda p}},
\end{align*}
where $w=se^{i\varphi}$.
Since for $\theta,\varphi\in[0,\pi]$,
\begin{eqnarray*}
|1-re^{i(\theta-\varphi)}|^{2}=1-2r\cos(\theta-\varphi)+r^2\le|1-re^{i(\theta+\varphi)}|^{2},
\end{eqnarray*}
we obtain
\begin{align}\label{kernel-p-mean-3}
\int_{-\pi}^{\pi}|K_{\lambda,\alpha}(z,se^{i\varphi})|^{p}\,dm_{\lambda}(\varphi)
\lesssim \int_{0}^{\pi}\frac{|1-rse^{i(\theta+\varphi)}|^{-2\lambda p}}{|1-rse^{i(\theta-\varphi)}|^{(\alpha+2)p}}\,dm_{\lambda}(\varphi).
\end{align}

In what follows, we consider the case for $\theta\in[0,\pi/2]$. If $\theta\in(\pi/2,\pi]$, the same procedure works by setting $\theta'=\pi-\theta$ and $\varphi'=\pi-\varphi$ since $K_{\lambda,\alpha}(re^{i\theta},se^{i\varphi})=\overline{K_{\lambda,\alpha}(re^{i\theta'},se^{i\varphi'})}$.

From (\ref{kernel-p-mean-3}) one has
\begin{align*}
\int_{-\pi}^{\pi}|K_{\lambda,\alpha}(z,se^{i\varphi})|^{p}\,dm_{\lambda}(\varphi)
\lesssim \int_{0}^{\pi/2}\left(\frac{|1-rse^{i(\theta+\varphi)}|^{-2\lambda p}}{|1-rse^{i(\theta-\varphi)}|^{(\alpha+2)p}}
+\frac{|1+rse^{i(\theta-\varphi)}|^{-2\lambda p}}{|1+rse^{i(\theta+\varphi)}|^{(\alpha+2)p}}\right)\,dm_{\lambda}(\varphi),
\end{align*}
and since, for $\theta,\varphi\in[0,\pi/2]$,
\begin{eqnarray*}
|1-re^{i(\theta-\varphi)}|^{2}\le|1+re^{i(\theta+\varphi)}|^{2}\quad\hbox{and}\quad |1-re^{i(\theta+\varphi)}|^{2}\le|1+re^{i(\theta-\varphi)}|^{2},
\end{eqnarray*}
it follows that
\begin{align*}
\int_{-\pi}^{\pi}|K_{\lambda,\alpha}(z,se^{i\varphi})|^{p}\,dm_{\lambda}(\varphi)
\lesssim \int_{0}^{\pi/2} \frac{|1-rse^{i(\theta+\varphi)}|^{-2\lambda p}}{|1-rse^{i(\theta-\varphi)}|^{(\alpha+2)p}}
\,dm_{\lambda}(\varphi).
\end{align*}
Furthermore, since for $\theta,\varphi\in[0,\pi/2]$,
\begin{align*}
|1-re^{i(\theta-\varphi)}|^{2}\asymp(1-r+|\theta-\varphi|)^2,\qquad
|1-re^{i(\theta+\varphi)}|^{2}\asymp(1-r+\theta+\varphi)^2,
\end{align*}
we get
\begin{align*}
\int_{-\pi}^{\pi}|K_{\lambda,\alpha}(z,se^{i\varphi})|^{p}\,dm_{\lambda}(\varphi)
\lesssim \int_{0}^{\pi/2} \frac{(1-rs+\theta+\varphi)^{-2\lambda p}\varphi^{2\lambda}}{(1-rs+|\theta-\varphi|)^{(\alpha+2)p}}
\,d\varphi.
\end{align*}

Now We split the last integral above into three parts, that are over $[0,\theta/2]$, $[\theta/2,3\theta/2]$ and $[3\theta/2,\pi/2]$ respectively. The first part is controlled by a multiple of
$$
\theta^{2\lambda+1}/(1-rs+\theta)^{(\alpha+2)p+2\lambda p}\lesssim (1-rs+\theta)^{2\lambda(1-p)}/(1-rs)^{(\alpha+2)p-1},
$$
the second part by
$$
\frac{\theta^{2\lambda}}{(1-rs+\theta)^{2\lambda p}} \int_{\theta/2}^{3\theta/2}\frac{d\varphi}{(1-rs+|\theta-\varphi|)^{(\alpha+2)p}}
\lesssim\frac{(1-rs+\theta)^{2\lambda(1-p)}}{(1-rs)^{(\alpha+2)p-1}},
$$
and the last part by
$$
\int_{3\theta/2}^{\pi/2} \frac{d\varphi}{(1-rs+\varphi)^{(\alpha+2)p-2\lambda(1-p)}}
\lesssim \frac{1}{(1-rs+\theta)^{(\alpha+2)p-2\lambda(1-p)-1}}
\lesssim\frac{(1-rs+\theta)^{2\lambda(1-p)}}{(1-rs)^{(\alpha+2)p-1}}.
$$
Collecting these estimates yields, for $\theta\in[0,\pi/2]$,
\begin{align*}
\int_{-\pi}^{\pi}|K_{\lambda,\alpha}(z,se^{i\varphi})|^{p}\,dm_{\lambda}(\varphi)
\lesssim \frac{(1-rs+|\sin\theta|)^{2\lambda(1-p)}}{(1-rs)^{(\alpha+2)p-1}}.
\end{align*}
As we interpreted earlier, the above conclusion is true for all $\theta\in[-\pi,\pi]$, and moreover, this is identical with (\ref{kernel-p-mean-2}), in view of the fact $|1-s^2z^2|\asymp 1-rs+|\sin\theta|$ for $z=re^{i\theta}\in\DD$, $s\in[0,1)$. The proof of the theorem is completed.
\end{proof}

\section{On the weighted spaces with weight $w(s)=(1-s)^{\alpha-1}$}

In this section we apply our main theorems in Section 4 to the weighted Bergman spaces, denoted by $A^{p}_{\lambda,\alpha}(\DD)$, with respect to the power weights $w(s)=(1-s)^{\alpha-1}$ for $\alpha>0$.
We shall characterize the dual space $(A^{p}_{\lambda,\alpha}(\DD))^*$ of $A^{p}_{\lambda,\alpha}(\DD)$ for $p_0\le p\le1$, and give a sufficient condition in terms of a Carleson type measure for a multiplication operator to be bounded from $A^{p_1}_{\lambda,\alpha}(\DD)$ into a unweighted  space $L_{\lambda}^{p_2}(\DD)$ for $p_0\le p_1\le1\le p_2<\infty$.

\subsection{The dual spaces of $A^{p}_{\lambda,\alpha}(\DD)$ for $p_0\le p\le1$}

As usual, the dual space $(A^{p}_{\lambda,w}(\DD))^*$ of a weighted $\lambda$-Bergman space $A^{p}_{\lambda,w}(\DD)$ consists of all linear continuous functionals on $A^{p}_{\lambda,w}(\DD)$, and the norm of ${\mathcal{L}}\in(A^{p}_{\lambda,w}(\DD))^*$ is defined by $\|{\mathcal{L}}\|=\sup_{\|f\|_{A^{p}_{\lambda,w}}=1}|{\mathcal{L}}(f)|$.

For the weight function $w(s)=(1-s)^{\alpha-1}$ with $\alpha>0$, the weighted $\lambda$-Bergman space $A^{p}_{\lambda,\alpha}(\DD)$ with $0<p<\infty$ is the set of $\lambda$-analytic functions $f$ in $\DD$ for which
$$
\|f\|_{A_{\lambda,\alpha}^{p}}:=\left(\int_{\DD}|f(z)|^{p}(1-|z|)^{\alpha-1}d\sigma_{\lambda}(z)\right)^{1/p}<\infty.
$$

\begin{lemma}\label{g-derivative-Bergman-b}
Let $a,b$ be two real numbers so that $0\le a<b\le a+1\le2$, and assume that $g$ is $\lambda$-analytic in $\DD$ and satisfies the condition
\begin{align}\label{derivative-estimate-8-1}
|D_{z}\left(zg(z)\right)|\lesssim \frac{|1-z^{2}|^a}{(1-|z|)^b}, \qquad z\in\DD.
\end{align}
If $a>0$, then $|g(z)|\lesssim 1$ when $0<b<1$, $1+|1-z^{2}|^a\log(|1-z^{2}|/(1-|z|))$ when $b=1$, $1+|1-z^{2}|^a/(1-|z|)^{b-1}$ when $1<b<a+1$, and $(|1-z^{2}|/(1-|z|))^a+\log(3/|1-z^{2}|)$ when $b=a+1$; and if $a=0$, then $|g(z)|\lesssim 1$ when $0<b<1$, $\log(2/(1-|z|))$ when $b=1$.
\end{lemma}

The proof of the lemma is based on a reproducing formula in terms of the ``derivative" given below.

\begin{lemma}\label{g-derivative-Bergman-a} {\rm (\cite[Lemma 6.7]{LW1})}
If $f$ is $\lambda$-analytic in $\DD$ and satisfies $(1-|z|^{2})^2D_{z}\left(zf(z)\right)\in L^{1}_{\lambda}(\DD)$, then
\begin{align}\label{reproducing-formula-8-1}
f(z)=\frac{\lambda+2}{2}\int_{\DD}K_{\lambda,1}(z,\overline{w})D_{w}\left(wf(w)\right)(1-|w|^{2})^2d\sigma_{\lambda}(w),\qquad z\in\DD,
\end{align}
where $K_{\lambda,1}$ is given by (\ref{Bergman-kernel-1}) with $\alpha=1$.
\end{lemma}

Now applying (\ref{Bergman-type-kernel-2}) and (\ref{derivative-estimate-8-1}) to (\ref{reproducing-formula-8-1}) we have
\begin{align*}
|f(z)|\lesssim \int_{\DD}\frac{|1-z^{2}|^a(1-|z|)^{2-b}}{(|1-z\overline{w}|+|1-zw|)^{2\lambda}} \frac{d\sigma_{\lambda}(w)}{|1-z\overline{w}|^3}.
\end{align*}
Let $z=re^{i\theta}$, $w=se^{i\varphi}$. Appealing to (\ref{elementary-equality-2}) and $|1-z\overline{w}|+|1-zw|\gtrsim |\sin\varphi|$, we get
\begin{align*}
|f(z)|\lesssim \int_0^1\int_{-\pi}^{\pi}\frac{(1-r+|\sin\theta|)^a(1-s)^{2-b}}{(1-rs+|\sin(\varphi-\theta)/2|)^3}\, d\varphi ds.
\end{align*}
According to this, the desired estimates in Lemma \ref{g-derivative-Bergman-b} will be obtained after an elementary process.

What we really use later is the following nearly trivial corollary of Lemma \ref{g-derivative-Bergman-b}.

\begin{corollary}\label{g-derivative-Bergman-c}
Let $a,b$ be two real numbers so that $0\le a<b\le a+1\le2$, and assume that $g$ is $\lambda$-analytic in $\DD$ and satisfies the condition (\ref{derivative-estimate-8-1}).
Then
\begin{align*}
|f(z)|\lesssim \frac{|1-z^{2}|^a}{(1-|z|)^b}, \qquad z\in\DD.
\end{align*}
\end{corollary}

The following theorem characterizes the dual space $(A^{p}_{\lambda,\alpha}(\DD))^*$ of $A^{p}_{\lambda,\alpha}(\DD)$ for $p_0\le p\le1$.

\begin{theorem} \label{dual-Bergman-space-a}
Let $p_0\le p\le1$, $\alpha>0$, and
$$
m=\left[\frac{\alpha+2\lambda+1}{p}-2\lambda\right].
$$
Then each ${\mathcal{L}}\in(A^{p}_{\lambda,\alpha}(\DD))^*$ is uniquely determined by a $\lambda$-analytic function $g$ on $\DD$, $g(z)=\sum_{k=0}^{\infty}b_{k}\phi_{k}^{\lambda}(z)$ says, satisfying the condition
\begin{align}\label{Bergman-functional-estimate-1}
\left|(D_z\circ z)^mg(z)\right|\lesssim \frac{|1-z^{2}|^{2\lambda(p^{-1}-1)}}{(1-|z|)^{m+1-(\alpha+1)/p}}, \qquad z\in\DD,
\end{align}
under the pairing duality given by
\begin{align*}
{\mathcal{L}}(f)=\sum_{k=0}^na_kb_k
\end{align*}
for all $\lambda$-analytic polynomials $f(z)=\sum_{k=0}^na_{k}\phi_{k}^{\lambda}(z)$, $n=0,1,\cdots$.
\end{theorem}

\begin{proof}
First note that $w(s)=(1-s)^{\alpha-1}$ with $\alpha>0$ satisfies the conditions (\ref{weight-condition-1-1}) and (\ref{weight-condition-1-2}) for all $q>\alpha$. In what follows we take
$$
q=(m+2\lambda+1)p-2\lambda-1
$$
and
\begin{eqnarray*}
G_{\lambda,\gamma}(z)=\sum_{n=0}^{\infty}
\frac{\Gamma(\lambda+1)\Gamma(n+\lambda+\gamma+2)}{\Gamma(\lambda+\gamma+2)\Gamma(n+\lambda+1)}\,
b_n\phi_{n}(z),\qquad z\in\DD,
\end{eqnarray*}
where $b_n={\mathcal{L}}(\phi_n)$ for $n=0,1,\cdots$.
By Theorem \ref{operator-boundedness-main-c}, a linear functional ${\mathcal{L}}$ on $A^{p}_{\lambda,\alpha}(\DD)$ is continuous if and only if $G_{\lambda,m-1}(z)$ defines a $\lambda$-analytic function on $\DD$ and satisfies  \begin{align}\label{G-function-norm-1}
|G_{\lambda,m-1}(z)|\lesssim \frac{|1-z^{2}|^{2\lambda(p^{-1}-1)}}{(1-|z|)^{m+1-(\alpha+1)/p}}, \qquad z\in\DD.
\end{align}

If we set $a=2\lambda(p^{-1}-1)$ and $b=m+1-(\alpha+1)/p$, it is easy to see that $0\le a<b\le a+1\le2$ for $p_0\le p\le1$ and $\alpha>0$. If $g$ satisfies (\ref{Bergman-functional-estimate-1}), then by Corollary \ref{g-derivative-Bergman-c}, $g\in A^{1}_{\lambda,b+1}(\DD)$, and by \cite[(59)]{LW1} and (\ref{phi-bound-1}), $|b_n|\lesssim n^{\lambda+b+1}$. Thus $G_{\lambda,m-1}(z)$ defines a $\lambda$-analytic function on $\DD$.
Since by (\ref{A-operator-4}),
\begin{align*}
\left|\partial_{\bar{z}}\phi_n(z)\right|\le\epsilon_n^{-1}n|z|^{n-1}\asymp n^{\lambda+1}|z|^{n-1}/\sqrt{\Gamma(2\lambda+1)},
\end{align*}
which allows us to take termwise differentiation $\partial_z$ in $\DD$ to $zg(z)$, by (\ref{Tzphi-2}) we have
\begin{align*}
(D_z\circ z)^mg(z)=\sum_{n=0}^{\infty}b_{n}(n+\lambda+1)^m\phi_{n}(z),\qquad z\in\DD.
\end{align*}
Now using a similar expansion to (\ref{elementary-equality-4}), with $m-1$ replacing $\alpha$ and an appropriate large $M$, gives
\begin{align*}
G_{\lambda,m-1}(z)=\sum_{j=0}^{M}\tilde{a}_{m-1,j}(D_z\circ z)^{m-j}g(z)+ g_M(z),
\end{align*}
where $g_M$ is a bounded $\lambda$-analytic function on $\DD$. The estimate (\ref{G-function-norm-1}) of $G_{\lambda,m-1}$ follows from (\ref{Bergman-functional-estimate-1}) and Corollary \ref{g-derivative-Bergman-c} immediately.

Conversely, if $G_{\lambda,m-1}$ satisfies (\ref{G-function-norm-1}), one can obtain the estimate (\ref{Bergman-functional-estimate-1}) for the function $g$ by the same process. The proof of the theorem is finished.
\end{proof}

\subsection{Multiplication operators on $A^{p}_{\lambda,\alpha}(\DD)$ for $p_0\le p\le1$}

For a given $\lambda$-analytic function $g$ on $\DD$, we consider the associated multiplication operator $T_g$ defined by
$$
T_gf(z)=g(z)f(z).
$$
In this subsection we shall give a sufficient condition of $g$ in terms of a Carleson type measure so that the multiplication operator $T_g$ becomes a bounded mapping from the weighted $\lambda$-Bergman spaces $A^{p_1}_{\lambda,\alpha}(\DD)$ into the unweighted  space $L_{\lambda}^{p_2}(\DD)$
for $p_0\le p_1\le1\le p_2<\infty$.

As usual, for given $\theta_{0}\in[0,2\pi)$ and $0<h<1$ define
$$
S_{\theta_{0}}(h)=\{z\in\DD:\,\,\, |z-e^{i\theta_{0}}|\le 2h\}.
$$

\begin{definition}
Let $d\mu$ be a finite Borel measure on the disc $\DD$. For $\lambda\ge0$, $\alpha>0$, we say $d\mu$ to be a $(\lambda,\alpha)$-Carleson measure if
\begin{align*}
\mu\left(S_{\theta}(h)\right)=O\left((h+|\sin\theta|)^{2\lambda}h^{\alpha}\right),\qquad 0\le\theta<2\pi,\,\, 0<h<1.
\end{align*}
 \end{definition}
If $\alpha=1$, $d\mu$ is called a $\lambda$-Carleson measure.

\begin{theorem} \label{multiplication-Bergman-a}
Let $p_0\le p_1\le1\le p_2<\infty$ and $\alpha>0$.
If the $\lambda$-analytic function $g$ on $\DD$ satisfies the condition that $d\mu_g(z):=|g(z)|^{p_2}d\sigma_{\lambda}(z)$ is a $(\lambda p_2/p_1,(\alpha+1)p_2/p_1)$-Carleson measure, then
the multiplication operator $T_g$ is bounded from $A^{p_1}_{\lambda,\alpha}(\DD)$ into $L_{\lambda}^{p_2}(\DD)$.
\end{theorem}

\begin{proof}
What we need to prove is $\|gf\|_{L_{\lambda}^{p_2}(\DD)}\lesssim \|f\|_{A_{\lambda,\alpha}^{p_1}}$ for $f\in A^{p_1}_{\lambda,\alpha}(\DD)$, that is,
\begin{align}\label{key-estimate-Carleson-0}
\left(\int_{\DD}|f(z)|^{p_2}d\mu_g(z)\right)^{1/p_2}
\lesssim
\left(\int_{\DD}|f(z)|^{p_1}(1-|z|)^{\alpha-1}d\sigma_{\lambda}(z)\right)^{1/p_1}, \qquad f\in A^{p_1}_{\lambda,\alpha}(\DD).
\end{align}
Again, $w(s)=(1-s)^{\alpha-1}$ with $\alpha>0$ satisfies the conditions (\ref{weight-condition-1-1}) and (\ref{weight-condition-1-2}) for all $q>\alpha$. We take
$$
q=(m+2\lambda+1)p_1-2\lambda-1,
$$
where
$$
m=\left[\frac{\alpha+2\lambda+1}{p_1}-2\lambda\right].
$$
Obviously $q>\alpha$, and since $\left(T_g\phi_n\right)(z)=g(z)\phi_n(z)$, by Theorem \ref{operator-boundedness-main-b} proving (\ref{key-estimate-Carleson-0}) is equivalent to showing
\begin{align}\label{key-estimate-Carleson-1}
\left(\int_{\DD}|K_{\lambda,m-1}(z,w)|^{p_2}d\mu_g(w)\right)^{1/p_2}
\lesssim
\frac{|1-z^{2}|^{2\lambda(p_1^{-1}-1)}}{(1-|z|)^{m+1-(\alpha+1)/p_1}},\qquad z\in\DD.
\end{align}

By Theorem \ref{Bergman-type-kernel-a}, we have
\begin{align}\label{key-estimate-Carleson-2}
\int_{\DD}|K_{\lambda,m-1}(z,w)|^{p_2}d\mu_g(w)
\lesssim J_1(z)+J_2(z),
\end{align}
where
\begin{align*}
J_1(z)= \int_{\DD} \frac{|1-z\overline{w}|^{-(m+1)p_2}\,d\mu_g(w)}{(|1-z\overline{w}|+|1-zw|)^{2\lambda p_2}},\qquad
J_2(z)= \int_{\DD} \frac{|1-zw|^{-(m+1)p_2}\,d\mu_g(w)}{(|1-z\overline{w}|+|1-zw|)^{2\lambda p_2}}.
\end{align*}

Let $z=re^{i\theta}$ with $3/4\le r<1$ and define the sets $\{E_k\}$ by $E_{-1}=\emptyset$, and for $k\ge0$,
\begin{align*}
E_{k}&=S_{\theta}(2^k(1-r)).
\end{align*}
Take a positive integer $N$ so that $2^N(1-r)\ge2$, and write
\begin{align*}
J_1(z)
\lesssim \sum_{k=0}^N\int_{E_k\setminus E_{k-1}} \frac{|1-z\overline{w}|^{-(m+1)p_2}}{(|1-z\overline{w}|+|1-zw|)^{2\lambda p_2}}\,d\mu_g(w).
\end{align*}

Since, with $w=se^{i\varphi}$,
\begin{align*}
|1-z\overline{w}|^2+|1-zw|^2&=2(1-rs)^2+4rs(1-\cos\theta\cos\varphi)\\
&\gtrsim\left(1-rs+|\sin\theta|+|\sin\varphi|\right)^2,
\end{align*}
we have
\begin{align*}
J_1(z)
&\lesssim \sum_{2^k(1-r)\le 1-r+|\sin\theta|} \int_{E_k\setminus E_{k-1}} \frac{|1-z\overline{w}|^{-(m+1)p_2}}{(1-r+|\sin\theta|)^{2\lambda p_2}}\,d\mu_g(w)\\
&\qquad\qquad +\sum_{2^k(1-r)>1-r+|\sin\theta|} \int_{E_k\setminus E_{k-1}} \frac{d\mu_g(w)}{|1-z\overline{w}|^{(m+1)p_2+2\lambda p_2}}.
\end{align*}
Note that for $w\in E_k\setminus E_{k-1}$ ($k\ge1$),
$$
|1-z\overline{w}|\ge2^{k-1}(1-r),
$$
and so
\begin{align*}
J_1(z)
&\lesssim \frac{(1-r)^{-(m+1)p_2}}{(1-r+|\sin\theta|)^{2\lambda p_2}}\sum_{2^k(1-r)\le 1-r+|\sin\theta|} \frac{\mu_g\left(E_k\right)}{2^{(m+1)p_2k}}\\
&\qquad\qquad +\sum_{2^k(1-r)>1-r+|\sin\theta|} \frac{\mu_g\left(E_k\right)}{(2^{k}(1-r))^{(m+1)p_2+2\lambda p_2}}.
\end{align*}

By the assumption that $d\mu_g(z):=|g(z)|^{p_2}d\sigma_{\lambda}(z)$ is a $(\lambda p_2/p_1,(\alpha+1)p_2/p_1)$-Carleson measure, it follows that
$$
\mu_g\left(E_k\right)\lesssim(2^k(1-r)+|\sin\theta|)^{2\lambda p_2/p_1}(2^k(1-r))^{(\alpha+1)p_2/p_1},
$$
and hence
\begin{align}\label{key-estimate-Carleson-3}
J_1(z)
&\lesssim \frac{(1-r+|\sin\theta|)^{2\lambda p_2(p_1^{-1}-1)}}{(1-r)^{(m+1)p_2-(\alpha+1)p_2/p_1}}
\sum_{2^k(1-r)\le 1-r+|\sin\theta|} \frac{1}{2^{[m+1-(\alpha+1)/p_1]p_2k}} \nonumber\\
&\qquad\qquad +\sum_{2^k(1-r)>1-r+|\sin\theta|} \frac{(1-r)^{-[m+2\lambda+1-(\alpha+2\lambda+1)/p_1]p_2}}{2^{[m+2\lambda+1-(\alpha+2\lambda+1)/p_1]p_2k}}.
\end{align}
But since
\begin{align*}
m+1-\frac{\alpha+1}{p_1}\ge m+2\lambda+1-\frac{\alpha+2\lambda+1}{p_1}>0,
\end{align*}
the first sum in (\ref{key-estimate-Carleson-3}) is bounded, and
the second sum is controlled by a multiple of
\begin{align*}
(1-r+|\sin\theta|)^{-[m+2\lambda+1-(\alpha+2\lambda+1)/p_1]p_2}
\lesssim\frac{(1-r+|\sin\theta|)^{2\lambda p_2(p_1^{-1}-1)}}{(1-r)^{(m+1)p_2-(\alpha+1)p_2/p_1}}.
\end{align*}
Thus we conclude
\begin{align*}
J_1(z)
\lesssim \frac{(1-r+|\sin\theta|)^{2\lambda p_2(p_1^{-1}-1)}}{(1-r)^{(m+1)p_2-(\alpha+1)p_2/p_1}}.
\end{align*}
Similarly one can get the same estimate for $J_2(z)$ as above, and substituting these into (\ref{key-estimate-Carleson-2}) proves (\ref{key-estimate-Carleson-1}) in view of (\ref{elementary-equality-1}).
\end{proof}

\section{Sequence multipliers on the weighted Bergman spaces}

In this section we consider the operators mapping the weighted $\lambda$-Bergman space $A^{p}_{\lambda,w}(\DD)$ for $p_0\le p\le1$ into the sequence space $\ell^s$ for $1\le s<\infty$ with norm $\|\{c_n\}\|_{\ell^s}:=\left(\sum_{s=0}^{\infty}|c_n|^s\right)^{1/s}$. A given sequence $\eta=\{\eta_n\}$ of complex numbers is said to a sequence multiplier from $A^{p}_{\lambda,w}(\DD)$ to $\ell^s$ if $T_{\eta}f:=\{\eta_n a_n\}\in\ell^s$ whenever  $f(z)=\sum_{n=0}^{\infty}a_{n}\phi_{n}^{\lambda}(z)\in A^{p}_{\lambda,w}(\DD)$. By the closed graph theorem, this is equivalent to the boundedness of the operator $T_{\eta}$, that is,
$$
\left(\sum_{n=0}^{\infty}|\eta_n a_n|^s\right)^{1/s}\lesssim\|f\|_{A^{p}_{\lambda,w}}.
$$

We shall need several lemmas.

\begin{lemma} \label{weight-estimate-a}
Let $w$ be a weight function satisfying the condition (\ref{weight-condition-1-2}) for some $q>0$. Then
$$
\int_0^1\frac{w(t)}{(1-rt)^q}dt\lesssim\frac{1}{(1-r)^q}\int_r^1w(s)ds,\qquad r\in(0,1).
$$
\end{lemma}

Indeed, dividing the integral on the left hand side above into two parts over $[0,r]$ and $[r,1)$ respectively, one has
\begin{align*}
\int_0^1\frac{w(t)}{(1-rt)^q}dt
\le \int_0^r\frac{w(t)}{(1-t)^q}dt+\frac{1}{(1-r)^q}\int_r^1w(s)ds.
\end{align*}
The lemma follows by applying the condition (\ref{weight-condition-1-2}) to the first integral on the right hand side.

\begin{lemma} \label{sequence-sum-a}
Let $w$ be a nonzero weight function satisfying the condition (\ref{weight-condition-1-2}) for some $q>0$, and let $a,b$ satisfy $b-q+1\ge a\ge0$. Then for a sequence $\{\alpha_n\}$ of nonnegative numbers, the following are equivalent:

{\rm (i)} for $r\in[0,1)$,
$$
\sum_{n=1}^{\infty}\alpha_nr^n\lesssim\frac{(1-r+|\sin\theta|)^a}{(1-r)^b}\int_r^1w(t)dt;
$$

{\rm (ii)} for $N\ge1$,
$$
\sum_{n=1}^N\alpha_n\lesssim (N^{-1}+|\sin\theta|)^a N^b\int_{1-N^{-1}}^1w(t)dt.
$$.
\end{lemma}

\begin{proof}
Assuming that part (i) holds and choosing $r=1-N^{-1}$ for $N\ge2$, part (ii) follows immediately since
$\sum_{n=1}^N\alpha_n\lesssim \sum_{n=1}^N\alpha_n(1-N^{-1})^n$.

Conversely assume that part (ii) holds. It follows that
\begin{align*}
\frac{1}{1-r}\sum_{n=1}^{\infty}\alpha_nr^n=\sum_{n=1}^{\infty}\left(\sum_{k=1}^n\alpha_k\right)r^n
\lesssim \sum_{n=1}^{\infty}(n^{-1}+|\sin\theta|)^a n^b\left(\int_{1-n^{-1}}^1w(t)dt\right)r^n.
\end{align*}
Noting that $\int_{1-n^{-1}}^1w(t)dt\lesssim\int_0^1w(t)t^ndt$ and $(n^{-1}+|\sin\theta|)^a\asymp n^{-a}+|\sin\theta|^a$, we have
\begin{align*}
\frac{1}{1-r}\sum_{n=1}^{\infty}\alpha_nr^n
&\lesssim \int_0^1w(t)\left(\sum_{n=1}^{\infty}n^{b-a}(rt)^n
+|\sin\theta|^a\sum_{n=1}^{\infty} n^b(rt)^n\right)dt\\
&\lesssim \int_0^1w(t)\left(\frac{1}{(1-rt)^{b-a+1}}
+\frac{|\sin\theta|^a}{(1-rt)^{b+1}}\right)dt\\
&\lesssim \left(\frac{1}{(1-r)^{b-a-q+1}}
+\frac{|\sin\theta|^a}{(1-r)^{b-q+1}}\right)\int_0^1\frac{w(t)}{(1-rt)^q}dt.
\end{align*}
This concludes part (i) by Lemma \ref{weight-estimate-a}.
\end{proof}

\begin{lemma} \label{sequence-sum-b}
Let $w$ be a nonzero weight function satisfying the conditions (\ref{weight-condition-1-1}) and (\ref{weight-condition-1-2}) for some $q>0$, and let $a,b$ satisfy $b-q\ge a\ge0$. Then for a sequence $\{\alpha_n\}$ of nonnegative numbers, the following three statements are equivalent:

{\rm (i)} for $N\ge1$,
$$
\sum_{n=1}^Nn^b\alpha_n
\lesssim (N^{-1}+|\sin\theta|)^a N^b\int_{1-N^{-1}}^1w(t)dt;
$$

{\rm (ii)} for $N\ge1$,
$$
\sum_{n=N}^{\infty}\alpha_n
\lesssim (N^{-1}+|\sin\theta|)^a \int_{1-N^{-1}}^1w(t)dt;
$$

{\rm (iii)} for $N\ge1$,
$$
\sum_{n=N}^{2N}\alpha_n\lesssim (N^{-1}+|\sin\theta|)^a \int_{1-N^{-1}}^1w(t)dt.
$$.
\end{lemma}

\begin{proof}
Assuming that part (i) holds and setting $s_n=\sum_{k=1}^nk^b\alpha_k$, for $M>N\ge1$ we have
\begin{align*}
\sum_{n=N}^M\alpha_n
&=\sum_{n=N}^{M-1}s_n(n^{-b}-(n+1)^{-b})+M^{-b}s_M-N^{-b}s_{N-1}\\
&\lesssim \sum_{n=N}^{M-1}n^{-1}(n^{-1}+|\sin\theta|)^a \int_{1-n^{-1}}^1w(t)dt
+(M^{-1}+|\sin\theta|)^a \int_{1-M^{-1}}^1w(t)dt.
\end{align*}
Letting $M\rightarrow\infty$ yields
\begin{align*}
\sum_{n=N}^{\infty}\alpha_n
\lesssim (N^{-1}+|\sin\theta|)^a \sum_{n=N}^{\infty}n^{-1}\int_{1-n^{-1}}^1w(t)dt.
\end{align*}
But since
\begin{align}\label{sequence-sum-1}
n^{-1}\int_{1-n^{-1}}^1w(t)dt
\lesssim \int_{1-(n-1)^{-1}}^{1-n^{-1}}\frac{1}{1-r}\int_{r}^1w(t)dtdr,
\end{align}
by the condition (\ref{weight-condition-1-1}) we obtain
\begin{align*}
\sum_{n=N}^{\infty}\alpha_n
&\lesssim (N^{-1}+|\sin\theta|)^a \left(N^{-1}\int_{1-N^{-1}}^1w(t)dt
+\sum_{n=N+1}^{\infty}\int_{1-(n-1)^{-1}}^{1-n^{-1}}w(r)dr\right)\\
&\lesssim (N^{-1}+|\sin\theta|)^a \int_{1-N^{-1}}^1w(t)dt.
\end{align*}
Thus part (ii) is proved under the assumption in part (i).

Conversely assume that part (ii) holds, and let $\tilde{s}_n=\sum_{k=n}^{\infty}\alpha_k$.
For $N\ge2$ we have
\begin{align*}
\sum_{n=1}^N n^b\alpha_n
&=\sum_{n=1}^N \tilde{s}_n(n^{b}-(n-1)^{b})-N^{b}\tilde{s}_{N+1}\\
&\lesssim \sum_{n=1}^N n^{b-1}(n^{-1}+|\sin\theta|)^a \int_{1-n^{-1}}^1w(t)dt.
\end{align*}
In a similar way to (\ref{sequence-sum-1}),
\begin{align*}
\sum_{n=1}^N n^b\alpha_n
&\lesssim \alpha_0+\sum_{n=2}^N \int_{1-(n-1)^{-1}}^{1-n^{-1}}\frac{(1-r+|\sin\theta|)^a}{(1-r)^{b+1}} \int_r^1w(t)dt dr,
\end{align*}
and then, applying the conditions (\ref{weight-condition-1-1}) gives
\begin{align*}
\sum_{n=1}^N n^b\alpha_n
&\lesssim \int_0^{1-N^{-1}}\left(\frac{1}{(1-r)^{b-a}}+ \frac{|\sin\theta|^a}{(1-r)^{b}}\right) w(r)dr\\
&\lesssim \left(N^{-a}+|\sin\theta|^a\right)N^{b-q} \int_0^{1-N^{-1}}
\frac{w(r)}{(1-r)^q} dr.
\end{align*}
Thus part (i) follows by means of the condition (\ref{weight-condition-1-2}).

That part (ii) implies part (iii) is obvious. It remains to show part (ii) under the assumption in part (iii).

For $N\ge1$, by part (iii) we have
\begin{align*}
\sum_{n=N}^{\infty}\alpha_n
\lesssim \sum_{k=0}^{\infty}\sum_{n=2^kN}^{2^{k+1}N}\alpha_n
\lesssim \left(N^{-1}+|\sin\theta|\right)^a \sum_{k=0}^{\infty}\int_{1-(2^kN)^{-1}}^1w(t)dt.
\end{align*}
Again in a similar way to (\ref{sequence-sum-1}),
\begin{align*}
\int_{1-(2^kN)^{-1}}^1w(t)dt
\lesssim \int_{1-(2^{k-1}N)^{-1}}^{1-(2^kN)^{-1}}\frac{1}{1-r}\int_{r}^1w(t)dtdr,
\end{align*}
and appealing to the conditions (\ref{weight-condition-1-1}), we get
\begin{align*}
\sum_{n=N}^{\infty}\alpha_n
&\lesssim \left(N^{-1}+|\sin\theta|\right)^a \left(\int_{1-N^{-1}}^1w(t)dt+\sum_{k=1}^{\infty}\int_{1-(2^{k-1}N)^{-1}}^{1-(2^kN)^{-1}}w(r)dr\right)\\
&\lesssim \left(N^{-1}+|\sin\theta|\right)^a \int_{1-N^{-1}}^1w(t)dt.
\end{align*}
Thus part (ii) is proved. The proof of the lemma is completed.
\end{proof}

\begin{lemma}\label{basis-sharp-estimate-a}
Assume that $\lambda>0$. We have
\begin{align*}
\left|\phi_{n}^{\lambda}(e^{i\theta})\right|\asymp
\left(|\sin\theta|+n^{-1}\right)^{-\lambda}\qquad\hbox{for}\,\,\, \theta\in[-\pi,\pi]\,\,\hbox{and}\,\,n\ge1.
\end{align*}
\end{lemma}

\begin{proof}
From Szeg\"o \cite[(7.33.6)]{Sz},
\begin{align*}
\left|P_{n}^{\lambda}(\cos\theta)\right|\lesssim n^{\lambda-1}\left(\sin\theta+n^{-1}\right)^{-\lambda}, \qquad \theta\in[0,\pi].
\end{align*}
Applying this to (\ref{basis-1}) and in view of (\ref{epsilon-bound-1}), we obtain
\begin{align*}
\left|\phi_{n}^{\lambda}(e^{i\theta})\right|\lesssim
\left(|\sin\theta|+n^{-1}\right)^{-\lambda}\qquad\hbox{for}\,\,\,\theta\in[-\pi,\pi]\,\,\hbox{and}\,\,n\ge1.
\end{align*}

The proof of the converse estimate needs a delicate analysis.
From \cite[(4.7.1) and (8.21.18)]{Sz}, for a given positive number $c_1$ we have
\begin{align*}
P_{n}^{\lambda}(\cos\theta)=\frac{2^{1-\lambda}}{\Gamma(\lambda)}\frac{n^{\lambda-1}}{(\sin\theta)^{\lambda}}
\left[\cos\left((n+\lambda)\theta-\frac{\lambda}{2}\pi\right) +O\left(\frac{1}{n\sin\theta}\right)\right], \,\,\,\theta\in\left[\frac{c_1}{n},\pi-\frac{c_1}{n}\right].
\end{align*}
Using this in (\ref{basis-1}) yields
\begin{align*}
\phi_{n}^{\lambda}(e^{i\theta})=\frac{|\sin\theta|^{-\lambda}}{\sqrt{2\pi\tilde{c}_{\lambda}}}
\left[e^{i\left((n+\lambda)\theta-\frac{\lambda}{2}({\rm sign}\,\theta)\pi\right)} +O\left(\frac{1}{n|\sin\theta|}\right)\right], \qquad\frac{c_1}{n}\le |\theta|\le\pi-\frac{c_1}{n}.
\end{align*}
Choosing $c_1$ appropriately large but fixed, we get
\begin{align}\label{basis-lower-bound-1}
\left|\phi_{n}^{\lambda}(e^{i\theta})\right|\gtrsim|\sin\theta|^{-\lambda}\asymp\left(|\sin\theta|+n^{-1}\right)^{-\lambda}, \qquad\frac{c_1}{n}\le |\theta|\le\pi-\frac{c_1}{n}.
\end{align}

To obtain a similar estimate to (\ref{basis-lower-bound-1}) in the ``complementary intervals" of $[-\pi,\pi]$, we appeal to a different method. Consider the function $h$ on $[-1,1]$ defined by, for $n\ge1$,
$$
h(x)=n(n+2\lambda)\left(P_{n}^{\lambda}(x)\right)^2+4\lambda^2(1-x^2)\left(P_{n-1}^{\lambda+1}(x)\right)^2,\qquad x\in[-1,1].
$$
In view of (\ref{epsilon-bound-1}) and (\ref{basis-1}), it is obvious that
\begin{align}\label{basis-estimate-1}
\left|\phi_{n}^{\lambda}(e^{i\theta})\right|\asymp n^{-\lambda}\sqrt{h(\cos\theta)}, \qquad\theta\in[-\pi,\pi].
\end{align}

Noting that $2\lambda P_{n-1}^{\lambda+1}(x)=P_{n}^{\lambda}(x)'$ (cf. \cite[(4.7.14)]{Sz}), we have
\begin{align*}
h'(x)=2P_{n}^{\lambda}(x)'\left[n(n+2\lambda)P_{n}^{\lambda}(x)+(1-x^2)P_{n}^{\lambda}(x)''\right]-2x\left(P_{n}^{\lambda}(x)'\right)^2;
\end{align*}
but by \cite[(4.7.5)]{Sz}, the expression in the square bracket above is identical with $(2\lambda+1)xP_{n}^{\lambda}(x)'$, so that
\begin{align*}
h'(x)=4\lambda x\left(P_{n}^{\lambda}(x)'\right)^2\ge0\qquad \hbox{for}\,\,\,x\in[0,1].
\end{align*}
Consequently the function $h(x)$ is nondecreasing on $[0,1]$.

Let $\theta^*$ and $\theta^{**}$ be the two adjacent zeros of $P_{n}^{\lambda}(\cos\theta)$ such that $\theta^{**}\le c_1/n<\theta^*$. It then follows that, for $0\le\theta\le c_1/n$,
\begin{align}\label{h-lower-bound-1}
h(\cos\theta)\ge h(\cos\theta^*)=4\lambda^2\sin^2\theta^*\left(P_{n-1}^{\lambda+1}(\cos\theta^*)\right)^2;
\end{align}
but by \cite[Theorem 8.9.1]{Sz}, $c_1/n<\theta^*\le c_2/n$ for some constant $c_2$($>c_1$) independent of $n$, and by \cite[(4.7.1), (4.7.14) and (8.9.2)]{Sz},
$$
2\lambda \left|P_{n-1}^{\lambda+1}(\cos\theta^*)\right|=\left|\left.\frac{d}{dx}P_{n}^{\lambda}(x)\right|_{x=\cos\theta^*}\right|\asymp n^{\lambda-\frac{1}{2}}\cdot n^{\lambda+\frac{3}{2}}=n^{2\lambda+1}.
$$
Combining these with (\ref{h-lower-bound-1}) gives us
\begin{align*}
h(\cos\theta)\ge h(\cos\theta^*)\asymp n^{4\lambda},\qquad 0\le\theta\le c_1/n,
\end{align*}
and from (\ref{basis-estimate-1}), we get $\left|\phi_{n}^{\lambda}(e^{i\theta})\right|\gtrsim n^{\lambda}$ for $0\le\theta\le c_1/n$, or equivalently,
\begin{align}\label{basis-lower-bound-2}
\left|\phi_{n}^{\lambda}(e^{i\theta})\right|\gtrsim \left(|\sin\theta|+n^{-1}\right)^{-\lambda}.
\end{align}

It is easy to check that (\ref{basis-lower-bound-2}) is true also for $-c_1/n\le\theta<0$ and $\pi-c_1/n\le|\theta|\le\pi$, and on account of (\ref{basis-lower-bound-1}), we have proved (\ref{basis-lower-bound-2}) for $\theta\in[-\pi,\pi]$ and for $n\ge1$.
The proof of Lemma \ref{basis-sharp-estimate-a} is finished.
\end{proof}

We now come to our main theorem in this section.

\begin{theorem}\label{sequence-multiplier-a}
Let $p_0\le p_1\le1\le p_2<\infty$ and $0<q<\infty$, and let $w$ be a nonzero weight function satisfying the conditions (\ref{weight-condition-1-1}) and (\ref{weight-condition-1-2}). Then a sequence $\eta=\{\eta_n\}$ of complex numbers is a sequence multiplier from $A^{p_1}_{\lambda,w}(\DD)$ to $\ell^{p_2}$, if and only if $\eta=\{\eta_n\}$ satisfies the condition
\begin{align}\label{sequence-multiplier-condition-1}
\left(\sum_{k=N}^{2N}|\eta_k|^{p_2}\right)^{p_1/p_2}\lesssim N^{(\lambda+1)p_1-2\lambda-1}\int_{1-N^{-1}}^1w(t)dt,
\qquad N=1,2,\cdots.
\end{align}
\end{theorem}

\begin{proof}
Assume that $\eta_0=0$ without loss of generality. Define the operator by $T_{\eta}f:=\{\eta_n a_n\}$ whenever  $f(z)=\sum_{n=0}^{\infty}a_{n}\phi_{n}^{\lambda}(z)\in A^{p_1}_{\lambda,w}(\DD)$; and in particular $T\phi_n^{\lambda}=\eta_n e_n$, where $\{e_0,e_1,\cdots\}$ is the canonic basis of $\ell^{p_2}$. Thus the associated $\ell^{p_2}$-valued function $F_{\lambda,\alpha}$ defined by (\ref{F-functin-1}) becomes
\begin{eqnarray*}
F_{\lambda,\alpha}(z)=\sum_{n=0}^{\infty}\eta_n\tau_n \phi_{n}(z)\,e_n,\qquad z\in\DD,
\end{eqnarray*}
where $\tau_n$ is given by (\ref{tau-notation-1}) and
\begin{align*}
\alpha=\frac{q+2\lambda+1}{p_1}-2\lambda-2,
\end{align*}
so that $\|F_{\lambda,\alpha}(z)\|_{\ell^{p_2}}=\left(\sum_{n=0}^{\infty}|\eta_n\tau_n \phi_{n}(z)|^{p_2}\right)^{1/p_2}$. Now by Theorem \ref{operator-boundedness-main-c} and the remark following it, the sequence $\eta=\{\eta_n\}$ is a sequence multiplier from $A^{p_1}_{\lambda,w}(\DD)$ to $\ell^{p_2}$, if and only if
\begin{align*}
\left(\sum_{n=0}^{\infty}|\eta_n\tau_n \phi_{n}(z)|^{p_2}\right)^{p_1/p_2}
\lesssim \frac{|1-z^{2}|^{2\lambda(1-p_1)}}{(1-|z|)^{(\alpha+2)p_1-1}}\int_{|z|^{p_1}}^1w(s)ds, \qquad z\in\DD;
\end{align*}
and by Lemma \ref{basis-sharp-estimate-a} and in view of (\ref{elementary-equality-1}) and (\ref{tau-notation-1}), this is equivalent to say
\begin{align}\label{sequence-multiplier-condition-2}
\sum_{n=0}^{\infty}\frac{\beta_n n^{(\alpha+1)p_1}|\eta_n|^{p_1}} {\left(n^{-1}+|\sin\theta|\right)^{\lambda p_1}}r^{p_1n}
\lesssim \frac{(1-r+|\sin\theta|)^{2\lambda(1-p_1)}}{(1-r)^{(\alpha+2)p_1-1}}\int_{r^{p_1}}^1w(s)ds, \qquad r\in[0,1),
\end{align}
for $\theta\in[-\pi,\pi]$ and for all $\{\beta_n\}\in\left(\ell^{p_2/p_1}\right)^*$ satisfying $\|\{\beta_n\}\|_{\left(\ell^{p_2/p_1}\right)^*}\le1$.

But according to Lemma \ref{sequence-sum-a}, with $a=2\lambda(1-p_1)$ and $b=(\alpha+2)p_1-1$, (\ref{sequence-multiplier-condition-2}) can be rephrased by
\begin{align*}
\sum_{n=1}^N\frac{\beta_n n^{(\alpha+1)p_1}|\eta_n|^{p_1}} {\left(n^{-1}+|\sin\theta|\right)^{\lambda p_1}}
\lesssim (N^{-1}+|\sin\theta|)^{2\lambda(1-p_1)} N^{(\alpha+2)p_1-1}\int_{1-N^{-1}}^1w(t)dt, \qquad N\ge1,
\end{align*}
and furthermore, according to Lemma \ref{sequence-sum-b}(iii), by
\begin{align*}
\sum_{n=N}^{2N}\frac{\beta_n n^{1-p_1}|\eta_n|^{p_1}} {\left(n^{-1}+|\sin\theta|\right)^{\lambda p_1}}
\lesssim (N^{-1}+|\sin\theta|)^{2\lambda(1-p_1)}\int_{1-N^{-1}}^1w(t)dt, \qquad N\ge1,
\end{align*}
or equivalently,
\begin{align*}
\sum_{n=N}^{2N}\beta_n|\eta_n|^{p_1}
\lesssim N^{p_1-1}(N^{-1}+|\sin\theta|)^{\lambda(2-p_1)}\int_{1-N^{-1}}^1w(t)dt,\qquad N\ge1.
\end{align*}
Finally, taking the supremum over all $\{\beta_n\}\in\left(\ell^{p_2/p_1}\right)^*$ satisfying $\|\{\beta_n\}\|_{\left(\ell^{p_2/p_1}\right)^*}\le1$, and the infimum over $\theta\in[-\pi,\pi]$,
we conclude that $\{\eta_n\}$ is a sequence multiplier from $A^{p_1}_{\lambda,w}(\DD)$ to $\ell^{p_2}$ if and only if (\ref{sequence-multiplier-condition-1}) holds for $\{\eta_n\}$. The proof of the theorem is completed.
\end{proof}

\section*{Declarations of interest: None}

\end{document}